\documentclass[a4paper,reqno]{amsart}
\usepackage[english]{babel}
\usepackage{amssymb,amsmath,hyperref}
\theoremstyle{plain}
\usepackage{bbold}
\usepackage{amsrefs}
\usepackage{stackrel}
\usepackage[foot]{amsaddr}
\usepackage{mathrsfs}
\usepackage{cleveref}

\usepackage{tikz}
\usetikzlibrary{arrows,chains,matrix,positioning,scopes}
\makeatletter
\tikzset{join/.code=\tikzset{after node path={%
\ifx\tikzchainprevious\pgfutil@empty\else(\tikzchainprevious)%
edge[every join]#1(\tikzchaincurrent)\fi}}}
\makeatother
\tikzset{>=stealth',every on chain/.append style={join},
         every join/.style={->}}
\tikzstyle{labeled}=[execute at begin node=$\scriptstyle,
   execute at end node=$]
\newtheorem{thm}{Theorem}

\newtheorem{cor}[thm]{Corollary}
\newtheorem{defi}[thm]{Definition}
\newtheorem{rem}[thm]{Remark}
\newtheorem{nota}[thm]{Notation}

\newtheorem{princ}[thm]{Principle}
\newtheorem{ack}[thm]{Acknowledgement}

\newtheorem{theme}[thm]{Theme}
\newtheorem{conj}[thm]{Conjecture}
\numberwithin{thm}{section}

\newcommand\be{\begin{equation}}
\newcommand\ee{\end{equation}}

\newbox\gnBoxA
\newdimen\gnCornerHgt
\setbox\gnBoxA=\hbox{$\ulcorner$}
\global\gnCornerHgt=\ht\gnBoxA
\newdimen\gnArgHgt

\def\Godelnum #1{%
	\setbox\gnBoxA=\hbox{$#1$}%
	\gnArgHgt=\ht\gnBoxA%
	\ifnum \gnArgHgt<\gnCornerHgt
		\gnArgHgt=0pt%
	\else
		\advance \gnArgHgt by -\gnCornerHgt%
	\fi
	\raise\gnArgHgt\hbox{$\ulcorner$} \box\gnBoxA %
		\raise\gnArgHgt\hbox{$\urcorner$}}
\def\bdefi{\begin{defi}\rm}
\def\edefi{\end{defi}}
\def\bnota{\begin{nota}\rm}
\def\enota{\end{nota}}
\def\brem{\begin{rem}\rm}
\def\erem{\end{rem}}

\def\FIVE{\Pi_{1}^{1}\text{-CA}_{0}}

\def\rel{\sqsubseteq}

\def\sler{\sqsupset}

\def\ATR{\textup{ATR}}
\def\UATR{\textup{UATR}}

\def\INT{\textup{INT}}

\def\USUP{\textup{USUP}}

\def\HB{\textup{HB}}
\def\RCA{\textup{RCA}}

\def\WO{\textup{WO}}
\def\RCAo{\textup{RCA}_{0}^{\omega}}
\def\RCAO{\textup{RCA}_{0}^{\Omega}}
\def\ATRO{\textup{ATR}_{\mathbb{o}}}
\def\UATRO{\textup{UATR}_{\mathbb{o}}}

\def\WKL{\textup{WKL}}

\def\SUP{\textup{SUP}}
\def\IVT{\textup{IVT}}
\def\UWKL{\textup{UWKL}}
\def\WEIMAX{\textup{WEIMAX}}
\def\WWKL{\textup{WWKL}}

\def\bye{\end{document}}

\def\N{{\mathbb  N}}

\def\R{{\mathbb  R}}

\def\EXP{\textup{EXP}}

\def\FAN{\textup{FAN}}
\def\UFAN{\textup{UFAN}}
\def\MUC{\textup{MUC}}

\def\R{{\mathbb{R}}}
\def\({\textup{(}}
\def\){\textup{)}}

\def\st{\textup{st}}
\def\asa{\leftrightarrow}

\def\di{\rightarrow}

\def\eps{\varepsilon}

\def\vx{\vec{x}}

\def\M{\mathcal{M}}

\def\ACA{\textup{ACA}}

\def\paai{\Pi_{1}^{0}\textup{-TRANS}}
\def\Paai{\Pi_{1}^{1}\textup{-TRANS}}

\def\QFAC{\textup{QF-AC}}

\def\PST{\textup{PST}}
\def\UPST{\textup{UPST}}

\def\UIVT{\textup{UIVT}}
\def\INT{\textup{int}}

\numberwithin{equation}{section}

\usepackage{comment}

\begin{document}
\title{Uniform and nonstandard existence in Reverse Mathematics}

\author{Sam Sanders}
\address{Department of Mathematics, Ghent University, Belgium \& Munich Center for Mathematical Philosophy, LMU Munich, Germany}
\email{sasander@cage.ugent.be}
\maketitle
\thispagestyle{empty}
\begin{abstract}
\emph{Reverse Mathematics} is a program in the foundations of mathematics which provides an elegant classification of theorems of ordinary mathematics 
based on computability.  Our aim is to provide an \emph{alternative} classification of theorems based on the central tenet of Feferman's \emph{Explicit Mathematics}, namely that \emph{a proof of existence of an object yields a procedure to compute 
said object}.  Our classification gives rise to the \emph{Explicit Mathematics theme} (EMT) of Nonstandard Analysis.   Intuitively speaking, the EMT states that
a standard object with certain properties can be computed by a functional if and only if this object exists classically \emph{with these same standard and nonstandard properties}.      
In this paper, we establish examples for the EMT ranging from the weakest to the strongest Big Five system of Reverse Mathematics.  
Our results are proved over the usual base theory of Reverse Mathematics, conservatively extended with higher types and Nelson's internal approach to Nonstandard Analysis.  
\end{abstract}

\vspace{-0.17cm}

\section{Introduction}\label{intro}
\subsection{Reverse Mathematics, explicitly}\label{introke}
The subject of this paper is the development of Reverse Mathematics (RM for short; See Section~\ref{RM} for an introduction) over a conservative extension of the usual base theory involving higher types and Nonstandard Analysis.  
This extended base theory, introduced in Section \ref{base}, is based on Nelson's \emph{internal set theory} (\cite{wownelly}) and Kohlenbach's \emph{higher-order} RM (\cite{kohlenbach2}).  
The aforementioned development of RM, 
which takes place in Section~\ref{revmathrcao}-\ref{zuslin}, leads to the formulation of the \emph{Explicit Mathematics Theme} (EMT for short), 
which we discuss now.  We follow the notations from Nelson's internal set theory.    
\begin{theme}[The theme from Explicit Mathematics]\rm
Consider a standard theorem of mathematics of the form:
\[
T^{\st}\equiv(\forall^{\st}x^{\sigma})(A^{\st}(x)\di (\exists^{\st} y^{\tau})B^{\st}(x,y)).
\]
The \emph{nonstandard} version of $T^{\st}$ is the statement:
\be\label{mf}\tag{$T^{*}$}
(\forall^{\st}x^{\sigma})(A^{\st}(x)\di (\exists^{\st} y^{\tau})B(x,y)), 
\ee
where $B^{\st}$ is `transferred' to $B$, i.e.\ the standardness predicate `st' is omitted.  
Furthermore, the \emph{uniform} version of $T$, is 
\be\tag{$UT$}
(\exists \Phi^{\sigma\di \tau})(\forall x^{\sigma})(A (x)\di B (x,\Phi(x))).
\ee
The \emph{Explicit Mathematics Theme} (EMT) is the observation that for many theorems $T$ as above, the base theory proves $T^{*}\asa UT$.  
\end{theme}
Note that the EMT expresses that the mere \emph{existence} of an object $y$ as in $T^{*}$, is 
equivalent to $y$ \emph{being computable via a functional} as in the\footnote{For certain statements $T$, there exists an additional interesting uniform version in which the existential quantifiers in $A$ are also removed by a functional.  When we encounter such statements, we shall make a distinction between $UT_{1}$ and $UT_{2}$.}  uniform version $UT$ (where the latter is free of `\st').  As suggested by its name, the EMT is inspired by the foundational program \emph{Explicit Mathematics}, whence discussed in Section \ref{EM}.   

\medskip

In this paper, we provide evidence for the EMT by establishing the latter for various theorems $T$ studied in second-order RM, where $T$ and $UT$ range from provable in the base theory to provable only in the strongest Big Five system.  
This development takes place in Sections \ref{revmathrcao} to \ref{zuslin}.  Finally, for some motivation regarding this study, we refer to Section \ref{mot}, while Sections \ref{RM} and \ref{EM} provide background information on Reverse and Explicit Mathematics.  
\subsection{Motivation}\label{mot}
We discuss the foundational significance of the EMT.  
\begin{enumerate} 
\item Central to the EMT is that statements involving \emph{higher-type} objects like $UT$ are equivalent to statements $T^{*}$ involving only \emph{lower-type} nonstandard objects.  
In this light, it seems incoherent to claim that higher-type objects are somehow `more real' than nonstandard ones (or vice versa).  
The EMT thus suggests that higher-order RM is implicit in Friedman-Simpson RM, as Nonstandard Analysis is used in the latter.  (See e.g.\ \cites{tahaar, tanaka1, keisler1, yo1, yokoyama2, yokoyama3, sayo, avi3, aloneatlast3}).  
Moreover, the EMT gives rise to an example of a higher-order statement \emph{implicit in Friedman-Simpson RM}, as discussed in Remark~\ref{loofer2}.       
\item In general, to prove $T^{*}\di UT$, one defines a functional $\Psi(\cdot, M)$ of (rather) elementary complexity, but involving an infinite number $M$.  
Assuming $T^{*}$, this functional is $\Omega$-invariant (See Definition \ref{homega}) and the axiom $\Omega$-CA from the base theory provides the required standard functional for $UT$.
As discussed in Section \ref{conrem}, these results can be viewed as a contribution to Hilbert's program for finitistic mathematics, as infinitary objects (the functional from $UT$) are decomposed into elementary objects.    
\item Fujiwara and Kohlenbach have established the equivalence between (classical) uniform existence as in $UT$ and intuitionistic provability for rather rich formulas classes (\cites{fuji1,fuji2}).  The EMT suggests that $T^{*}$ constitutes another way of capturing intuitionistic provability.  Nonetheless, we establish the EMT for statements beyond the Fujiwara-Kohlenbach metatheorems.     
\item Our results reinforce the heuristic $\frac{\WKL_{0}}{\ACA_{0}}\approx \frac{\ATR_{0}}{\FIVE}$ put forward in \cite{simpson2}*{I.11.7}.  
In particular, our treatment of the EMT for the fan theorem, i.e.\ the classical contraposition of WKL, and for $\ATR_{0}$ are \emph{neigh identical}.  
Furthermore, the EMT for $(S^{2})$, the functional version of $\FIVE$, is highly similar to the EMT for $(\exists^{2})$, the functional version of $\ACA_{0}$, thanks to the bounding result in Theorem \ref{rep}.  
In conclusion, Nonstandard Analysis allows us to treat sets of numbers in the same way as one treats numbers.  
\end{enumerate}
Besides the previous arguments, a first general motivation for the study of higher-order RM is as follows: It was shown in \cite{samimplicit} that higher-order statements are implicit in second-order RM.  
A second motivation, based on Feferman's \emph{Explicit Mathematics} (See Section \ref{EM}) is discussed in Sections~\ref{RM} and~\ref{EM}.  
Finally, we urge the reader to first consult Remark \ref{ohdennenboom} so as to clear up a common misconception regarding Nelson's framework.  

\subsection{Reverse Mathematics: a `computable' classification}\label{RM}
Reverse Mathematics (RM) is a program in the foundations of mathematics initiated around 1975 by Friedman (\cites{fried,fried2}) and developed extensively by Simpson (\cite{simpson2, simpson1}) and others.  
The aim of RM is to find the axioms necessary to prove a statement of \emph{ordinary} mathematics, i.e.\ dealing with countable or separable objects.   
The classical\footnote{In \emph{Constructive Reverse Mathematics} (\cite{ishi1}), the base theory is based on intuitionistic logic.  
} base theory $\RCA_{0}$ of `computable\footnote{The system $\RCA_{0}$ consists of induction $I\Sigma_{1}$, and the {\bf r}ecursive {\bf c}omprehension {\bf a}xiom $\Delta_{1}^{0}$-CA.} mathematics' is always assumed.  Thus, the aim is to find the minimal axioms $A$ to derive a given statement $T$ in $\RCA_{0}$;  In symbols:  
\begin{quote}
\emph{The aim of \emph{RM} is to find the minimal axioms $A$ such that $\RCA_{0}\vdash [A\di T]$ for statements $T$ of ordinary mathematics.}
\end{quote}
Surprisingly, once the minimal axioms $A$ have been found, we almost always also have $\RCA_{0}\vdash [A\asa T]$, i.e.\ not only can we derive the theorem $T$ from the axioms $A$ (the `usual' way of doing mathematics), we can also derive the axiom $A$ from the theorem $T$ (the `reverse' way of doing mathematics).  In light of the latter, the discipline was baptised `Reverse Mathematics'.    

\medskip

In the majority\footnote{Exceptions are classified in the so-called Reverse Mathematics Zoo (\cite{damirzoo}).
} 
of cases, for a statement $T$ of ordinary mathematics, either $T$ is provable in $\RCA_{0}$, or the latter proves $T\asa A_{i}$, where $A_{i}$ is one of $\WKL_{0}, \ACA_{0},$ $ \ATR_{0}$ or $\FIVE$.  The latter together with $\RCA_{0}$ form the `Big Five' and the aforementioned observation that most mathematical theorems fall into one of the Big Five categories, is called the \emph{Big Five phenomenon} (\cite{montahue}*{p.\ 432}).  
Furthermore, each of the Big Five has a natural formulation in terms of (Turing) computability (See e.g.\ \cite{simpson2}*{I.3.4, I.5.4, I.7.5}).
As noted by Simpson in \cite{simpson2}*{I.12}, each of the Big Five also corresponds (loosely) to a foundational program in mathematics.  
\medskip

An alternative view of Reverse Mathematics is as follows (and expressed in part by \cite{simpson2}*{Remark I.8.9.5}):  Reverse Mathematics studies theorems of mathematics `as they stand', instead of the common practice in constructive mathematics of introducing extra (often perceived as unnatural) conditions to make these theorems provable constructively.  In other words, rather than \emph{enforcing computability} via extra conditions, RM takes a `relative' stance:  Assuming computable mathematics in the guise of $\RCA_{0}$, how non-computable is a given theorem of mathematics, as measured by which of the other Big Five (or other principles) it is equivalent to?   

\medskip

In conclusion, Reverse Mathematics can be viewed as a classification of theorems of ordinary mathematics from the point of view of \emph{computability} (See e.g.\ \cite{simpson2}*{I.3.4}). 
A natural question is if there are \emph{other interesting ways} of classifying these theorems, which is part of the motivation of this paper, and discussed next.  

\subsection{Explicit Mathematics}\label{EM}
Around 1967, Bishop introduced \emph{Constructive Analysis} (\cite{bish1}), an approach to mathematics with a strong focus on computational meaning, but compatible with classical, recursive, and intuitionistic mathematics.  
In order to provide a natural formalisation for Constructive Analysis, Feferman introduced \emph{Explicit Mathematics} (EM) in \cites{feferman2,fefmar,fefmons}.  To capture Bishop's constructive notion of existence in a classical-logic setting, EM is built around the central tenet: 
\begin{center}  
\emph{A proof of existence of an object yields a procedure to compute said object}.
\end{center}
Soon after its inception, it was realised that the framework of EM is quite flexible and can be used for the study of much more general topics, such as reductive proof theory, generalised recursion theory, type theory and programming languages, et cetera.
A monograph on the topic of EM is forthcoming as \cite{fefermaninf}.  

\medskip

Similar to the second view of Reverse Mathematics (as classifying theorems based on computability, rather than `forcing computability onto theorems'), one can approach EM from the same relative point of view:  
Rather than enforcing the central tenet of EM, one can ask the following `relative' question:  
\begin{center}
\emph{For a given theorem $T$, what extra axioms are needed to compute the objects claimed to exist by $T$?}
\end{center}
In other words, how strong is $UT$ the uniform version of a theorem $T$?  To this question, the EMT from Section \ref{introke} provides a surprising(ly) uniform answer.  
\section{About and around the base theory $\RCAO$}\label{base}
In this section, we introduce the base theory $\RCAO$ in which we will prove our results.  
We discuss some basic results and introduce some notation.
\subsection{The system $\RCAO$}
In two words, $\RCAO$ is a conservative extension of Kohlenbach's base theory $\RCAo$ from \cite{kohlenbach2} with certain axioms from Nelson's \emph{Internal Set Theory} (\cite{wownelly}) based on the approach from \cites{brie,bennosam}.    
This conservation result is proved in \cite{bennosam}, while certain partial results are implicit in \cite{brie}.  In turn, $\RCAo$ is a conservative extension of $\RCA_{0}$ for the second-order language by \cite{kohlenbach2}*{Prop.\ 3.1}.  

\medskip

In Nelson's \emph{syntactic} approach to Nonstandard Analysis (\cite{wownelly}), as opposed to Robinson's semantic one (\cite{robinson1}), a new predicate `st($x$)', read as `$x$ is standard' is added to the language of ZFC.  
The notations $(\forall^{\st}x)$ and $(\exists^{\st}y)$ are short for $(\forall x)(\st(x)\di \dots)$ and $(\exists y)(\st(y)\wedge \dots)$.
The three axioms \emph{Idealization}, \emph{Standard Part}, and \emph{Transfer} govern the new predicate `st'  and give rise to a conservative extension of ZFC.   
Nelson's approach has been studied in the context of higher-type arithmetic in e.g.\ \cite{brie, bennosam, avi3}.

\medskip

Following Nelson's approach in arithmetic, we define $\RCAO$ as the system 
\[
\textup{E-PRA}_{\st}^{\omega*}+\textup{QF-AC}^{1,0} +\textup{HAC}_{\textup{int}}+\textup{I}+ \textup{PF-TP}_{\forall} 
\]
from \cite{bennosam}*{\S3.2-3.3}.  
To guarantee that $\RCAO$ is a conservative extension of $\RCAo$, Nelson's axiom \emph{Standard part} must be limited to $\Omega$-CA defined below (which derives from HAC$_{\INT}$), 
while Nelson's axiom \emph{Transfer} has to be limited to universal formulas \emph{without} parameters, as in PF-TP$_{\forall}$.  We have the following theorem  
\begin{thm}
The system \textup{E-PRA$_{\st}^{\omega*}+\textup{HAC}_{\textup{int}}+\textup{I}+ \textup{PF-TP}_{\forall}$} is a conservative extension of \textup{E-PRA}$^{\omega}$.  
The system $\RCAO$ is a $\Pi_{2}^{0}$-conservative extension of $\textup{PRA}$.  
\end{thm}
\begin{proof}
See \cite{bennosam}*{Cor.\ 9}.  
\end{proof}
The conservation result for $\textup{E-PRA}_{\st}^{\omega*}+\textup{QF-AC}^{1,0}$ is trivial.  
Furthermore, omitting PF-TP$_{\forall}$, the theorem is implicit in \cite{brie}*{Cor.\ 7.6} as the proof of the latter goes through as long as EFA is available.
We now discuss the new axioms in more detail.  
\subsection{The Transfer principle of $\RCAO$}
We first discuss the \emph{Transfer principle} included in $\RCAO$, which is as follows.     
\begin{princ}[PF-TP$_\forall$]  
For any internal formula $\varphi(x^{\tau})$ with all parameters shown, we have $(\forall^{\st}x^{\tau})\varphi(x)\di (\forall x)\varphi(x)$.
\end{princ} 
A special case of the previous can be found in Avigad's system NPRA$^{\omega}$ from \cite{avi3}.  
The omission of parameters in PF-TP$_{\forall}$ is essential, as is clear from Theorem \ref{markje}, for which we introduce:
\be\tag{$\paai$}
 (\forall^{\st}f^{1})\big[(\forall^{\st}n)f(n)=0\di (\forall n)f(n)=0],
\ee 
\be\tag{$\exists^{2}$}
(\exists \varphi^{2})(\forall g^{1})\big[(\exists x^{0})g(x)=0 \asa \varphi(g)=0  \big].
\ee
Note that standard parameters are allowed in $f$, and that $(\exists^{2})$ is the functional version of $\ACA_{0}$ (\cite{simpson2}*{III}), i.e.\ arithmetical comprehension.
\begin{thm}\label{markje}
The system $\RCAO$ proves $\paai\asa (\exists^{2})$.
\end{thm}
\begin{proof}
By \cite{bennosam}*{Cor 12}.  We sketch part of the proof in Theorem \ref{trikke} below.
\end{proof}
Besides being essential for the proof of the previous theorem, PF-TP$_{\forall}$ implies that all functionals defined \emph{without parameters} are standard, as discussed next. 
\begin{rem}[Standard functionals]\label{tokkiep}\rm
We discuss an important advantage of the axiom PF-TP$_{\forall}$.  First of all, given the existence of a functional, like e.g.\ the existence of the fan functional (See e.g.\ \cite{noortje, kohlenbach2}) as follows:
\be\label{MUC2}\tag{MUC}
(\exists \Omega^{3})(\forall \varphi^{2}) (\forall f^{1}, g^{1}\leq_{1}1 )[\overline{f}(\Omega(\varphi))=_{0}\overline{g}(\Omega(\varphi))\di \varphi(f)=_{0}\varphi(g)],
\ee  
we immediately obtain, via the contraposition of PF-TP$_{\forall}$, that 
\be\label{drifsd}
(\exists^{\st} \Theta^{3})(\forall \varphi^{2}) (\forall f^{1}, g^{1}\leq_{1}1 )[\overline{f}(\Theta(\varphi))=_{0}\overline{g}(\Theta(\varphi))\di \varphi(f)=_{0}\varphi(g)].
\ee
In other words, we may assume that the fan functional is \emph{standard}.  The same holds for \emph{any functional} of which the definition does not involve additional parameters.   

\medskip

Secondly, we may assume $\Omega(\varphi)$ is the \emph{least number} as in \eqref{MUC2}, which implies that $\Theta(\varphi)$ from \eqref{drifsd} can also be assumed to have this property.  
However, then $\Theta(\varphi)=_{0}\Omega(\varphi)$ for any $\varphi^{2}$, implying $\Theta=_{3}\Omega$, i.e.\ if it exists, the fan functional is \emph{unique and standard}.  
The same again holds for any uniquely-defined functional of which the definition does not involve additional parameters.       
\end{rem}
The previous observation prompted the addition to $\RCAO$ of axioms reflecting the uniqueness and standardness of certain functionals (See \cite{bennosam}*{\S3.3}).  
It should be noted that Nelson makes a similar observation concerning IST in \cite{wownelly}*{p.\ 1166}.

\subsection{The Standard part principle of $\RCAO$}\label{STPSEC}
Next, we discuss the \emph{Standard Part principle}, called $\Omega$-CA, included in $\RCAO$.  
Intuitively speaking, a Standard Part principle allows us to convert nonstandard into standard objects.  

\medskip

By way of example,  the following type 1-version of the Standard part principle results in a conservative extension of $\WKL_{0}$ (See \cites{keisler1, briebenno}).  
\be\label{STP}\tag{STP}
(\forall X^{1})(\exists^{\st} Y^{1})(\forall^{\st} x^{0})(x\in X\asa x\in Y).
\ee
Here, we have used set notation to increase readability;  We assume that sets $X^{1}$ are given by their characteristic functions $f^{1}_{X}$, i.e.\ $(\forall x^{0})[x\in X\asa f_{X}(x)=1]$.     
The set $Y$ from \eqref{STP} is also called the \emph{standard part} of $X$. 

\medskip
    
%
%
We now discuss the Standard Part principle $\Omega$-CA, a very practical consequence of the axiom HAC$_{\textup{int}}$.  
Intuitively speaking, $\Omega$-CA expresses that we can obtain 
the standard part (in casu $G$) of \emph{$\Omega$-invariant} nonstandard objects (in casu $F(\cdot,M)$).   
Note that we write `$N\in \Omega$' as short for $\neg\st(N^{0})$.
\bdefi[$\Omega$-invariance]\label{homega} Let $F^{(\sigma\times  0)\di 0}$ be standard and fix $M^{0}\in \Omega$.  
Then $F(\cdot,M)$ is {\bf $\Omega$-invariant} if   
\be\label{homegainv}
(\forall^{\st} x^{\sigma})(\forall N^{0}\in \Omega)\big[F(x ,M)=_{0}F(x,N) \big].  
\ee
\edefi
\begin{princ}[$\Omega$-CA]\rm Let $F^{(\sigma\times 0)\di 0}$ be standard and fix $M^{0}\in \Omega$.
For every $\Omega$-invariant $F(\cdot,M)$, there is a standard $G^{\sigma\di 0}$ such that
\be\label{homegaca}
(\forall^{\st} x^{\sigma})(\forall N^{0}\in \Omega)\big[G(x)=_{0}F(x,N) \big].  
\ee
\end{princ}
The axiom $\Omega$-CA provides the standard part of a nonstandard object, if the latter is \emph{independent of the choice of infinite number} used in its definition.  
\begin{thm}\label{drifh}
The system $\RCAO$ proves $\Omega\textup{-CA}$.  
\end{thm}
\begin{proof}
See e.g.\ \cite{tale};  We also sketch the derivation of $\Omega$-CA from HAC$_{\INT}$.  The latter takes the form  
\be\tag{$\textup{HAC}_{\INT}$}
(\forall^{\st} x)(\exists^{\st} y)\varphi(x,y)\di (\exists^{\st} \Phi)(\forall^{\st}x)(\exists y\in F(x))\varphi(x,y),
\ee
where $\varphi(x,y)$ is internal, i.e.\ not involving the standardness predicate `$\st$', and where $\Phi(x)$ is a \emph{finite sequence} of objects of the type of $y$. 
Thus, HAC$_{\textup{int}}$ does not provide a witness to $y$, but a \emph{sequence of} possible witnesses.  

\medskip

Let $F(\cdot,M^{0})$ is $\Omega$-invariant, i.e.\ we have 
\be\label{dorkillllll}
(\forall^{\st} x^{\sigma})(\forall N^{0},M^{0}\in \Omega)\big[F(x ,M)=_{0}F(x,N) \big].  
\ee
We immediately obtain (any infinite $k^{0}$ will do) that 
\[
(\forall^{\st} x^{\sigma})(\exists k^{0})(\forall N^{0},M^{0}\geq k)\big[F(x ,M)=_{0}F(x,N) \big].  
\]
By the induction axioms present in $\RCAO$, there is a least such $k$ for every standard $x^{\sigma}$.  
By our assumption \eqref{dorkillllll}, such \emph{least} number $k^{0}$ must be standard, yielding:
\[
(\forall^{\st} x^{\sigma})(\exists^{\st}k^{0})(\forall N^{0},M^{0}\geq k)\big[F(x ,M)=_{0}F(x,N) \big].
\]
Now apply HAC$_{\textup{int}}$ to obtain standard $\Phi^{\sigma\di 0}$ such that
\[
(\forall^{\st} x^{\sigma})(\exists k^{0}\in \Phi(x))(\forall N^{0},M^{0}\geq k)\big[F(x ,M)=_{0}F(x,N) \big].
\]
Next, define $\Psi(x):= \max_{i<|\Phi(x)|}\Phi(x)(i)$ and note that 
\[
(\forall^{\st} x^{\sigma})(\forall N^{0},M^{0}\geq \Psi(x))\big[F(x ,M)=_{0}F(x,N) \big].
\]
Finally, put $G(x):=F(x,\Psi(x))$ and note that $\Omega$-CA follows.
\end{proof}
In light of the previous proof, one easily establishes the following corollaries.
\begin{cor}\label{genall}
In $\RCAO$, we have for all standard $F^{(\sigma\times 0)\di 1}$ that
\begin{align*}
(\forall^{\st} x^{\sigma})(\forall M,N \in & \Omega)\big[F(x ,M)\approx_{1}F(x,N) \big] \\
&\di (\exists^{\st}G^{\sigma\di 1})(\forall^{\st} x^{\sigma})(\forall N^{0}\in \Omega)\big[G(x)\approx_{1}F(x,N) \big],
\end{align*}
where $f^{1}\approx_{1} g^{1}$ if $(\forall^{\st}n^{0})(f(n)=_{0}g(n))$. 
\end{cor}
\begin{cor}\label{genalli}
In $\RCAO$, for all standard $F^{(\sigma\times 0)\di 1}$ and internal formulas $C$,
\begin{align*}
(\forall^{\st} x^{\sigma})(\forall M,N \in & \Omega)\big[C(F,x)\di F(x ,M)\approx_{1}F(x,N) \big] \\
&\di (\exists^{\st}G^{\sigma\di 1})(\forall^{\st} x^{\sigma})(\forall N^{0}\in \Omega)\big[C(F,x)\di  G(x)\approx_{1}F(x,N) \big].
\end{align*}
\end{cor}
Applications of the previous corollaries are assumed to be captured under the umbrella-term `$\Omega$-CA'.  
Furthermore, by the above, if we drop the $\Omega$-invariance condition in $\Omega$-CA, the resulting system is a non-conservative extension of $\RCAo$.       
\subsection{Notations and remarks}
We finish this section with some remarks.  First of all, we shall use the same notations as in \cite{bennosam}, some of which we repeat here.  
\begin{rem}[Notations]\label{notawin}\rm
We write $(\forall^{\st}x^{\tau})\Phi(x^{\tau})$ and $(\exists^{\st}x^{\sigma})\Psi(x^{\sigma})$ as short for 
$(\forall x^{\tau})\big[\st(x^{\tau})\di \Phi(x^{\tau})\big]$ and $(\exists^{\st}x^{\sigma})\big[\st(x^{\sigma})\wedge \Psi(x^{\sigma})\big]$.     
We also write $(\forall x^{0}\in \Omega)\Phi(x^{0})$ and $(\exists x^{0}\in \Omega)\Psi(x^{0})$ as short for 
$(\forall x^{0})\big[\neg\st(x^{0})\di \Phi(x^{0})\big]$ and $(\exists x^{0})\big[\neg\st(x^{0})\wedge \Psi(x^{0})\big]$.  Furthermore, if $\neg\st(x^{0})$ (resp.\ $\st(x^{0})$), we also say that $x^{0}$ is `infinite' (resp.\ finite) and write `$x^{0}\in \Omega$'.  
Finally, a formula $A$ is `internal' if it does not involve $\st$, and $A^{\st}$ is defined from $A$ by appending `st' to all quantifiers (except bounded number quantifiers).    
\end{rem}
We will use the usual notations for rational and real numbers and functions as introduced in \cite{kohlenbach2}*{p.\ 288-289} (and \cite{simpson2}*{I.8.1} for the former).  
\begin{rem}[Real number]\label{keepintireal}\rm
A (standard) real number $x$ is a (standard) fast-converging Cauchy sequence $q_{(\cdot)}^{1}$, i.e.\ $(\forall n^{0}, i^{0})(|q_{n}-q_{n+i})|<_{0} \frac{1}{2^{n}})$.  
We freely make use of Kohlenbach's `hat function' from \cite{kohlenbach2}*{p.\ 289} to guarantee that every sequence $f^{1}$ can be viewed as a real.  
Two reals $x, y$ represented by $q_{(\cdot)}$ and $r_{(\cdot)}$ are \emph{equal}, denoted $x=y$, if $(\forall n)(|q_{n}-r_{n}|\leq \frac{1}{2^{n}})$. Inequality $<$ is defined similarly.         
We also write $x\approx y$ if $(\forall^{\st} n)(|q_{n}-r_{n}|\leq \frac{1}{2^{n}})$ and $x\gg y$ if $x>y\wedge x\not\approx y$.  Functions $F$ mapping reals to reals are represented by functionals $\Phi^{1\di 1}$ such that $(\forall x, y)(x=y\di \Phi(x)=\Phi(y))$, i.e.\ equal reals are mapped to equal reals.   
\end{rem}
As hinted at by Corollary \ref{genall}, the notion of equality in $\RCAO$ is important.  
\begin{rem}[Equality]\label{equ}\rm
The system $\RCAo$ only includes equality between natural numbers `$=_{0}$' as a primitive.  Equality `$=_{\tau}$' for type $\tau$-objects $x,y$ is then defined as follows:
\be\label{aparth}
[x=_{\tau}y] \equiv (\forall z_{1}^{\tau_{1}}\dots z_{k}^{\tau_{k}})[xz_{1}\dots z_{k}=_{0}yz_{1}\dots z_{k}]
\ee
if the type $\tau$ is composed as $\tau\equiv(\tau_{1}\di \dots\di \tau_{k}\di 0)$.
In the spirit of Nonstandard Analysis, we define `approximate equality $\approx_{\tau}$' as follows:
\be\label{aparth2}
[x\approx_{\tau}y] \equiv (\forall^{\st} z_{1}^{\tau_{1}}\dots z_{k}^{\tau_{k}})[xz_{1}\dots z_{k}=_{0}yz_{1}\dots z_{k}]
\ee
with the type $\tau$ as above.  
The system $\RCAo$ includes the axiom of extensionality for all $\varphi^{\rho\di \tau}$ as follows:
\be\label{EXT}\tag{E}  
(\forall  x^{\rho},y^{\rho}) \big[x=_{\rho} y \di \varphi(x)=_{\tau}\varphi(y)   \big].
\ee
However, as noted in \cite{brie}*{p.\ 1973}, the so-called axiom of standard extensionality \eqref{EXT}$^{\st}$ is problematic and cannot be included in $\RCAO$.  

\medskip

Now, certain functionals (like $(\exists^{2})^{\st}$ introduced above) are standard extensional, but others may not be.  
Hence, we shall sometimes prepend `E-' to an axiom defining a functional to express that the latter is standard extensional.  In particular, if $\textup{AX}^{\st}\equiv (\exists^{\st}\Phi^{\rho\di \tau})A^{\st}(\Phi)$ for internal $A$, then 
$\textup{E-AX}^{\st}$ is AX$^{\st}$ plus the statement that $\Phi$ as in the latter satisfies standard extensionality as in \eqref{EXT}$^{\st}$. 
\end{rem}
As an example illustrating the previous remark, the functional from UWKL$^{\st}$ is standard extensional if it outputs the left-most path in an infinite binary tree, and this property naturally emerges in the proof of Theorem \ref{trikke}.
In Theorem \ref{boon} we show that standard extensionality as in E-UWKL$^{\st}$ follows naturally from UWKL.  

\medskip

In light of Corollary \ref{genall}, it is obvious how $\Omega$-CA can be further generalised to $F^{(\sigma\times 0)\di \tau}$ using `$\approx_{\tau}$' instead of `$\approx_{1}$';  The same holds for `$\approx$' and real-valued $F$.    
%


\begin{rem}[The computable nature of operations in $\RCAO$]\label{ohdennenboom}\rm
Tennenbaum's theorem (\cite{kaye}*{\S11.3}) `literally' states that any nonstandard model of PA is not computable.  \emph{What is meant} is that for a nonstandard model $\M$ of PA, the operations $+_{\M}$ and $\times_{\M}$ cannot be computably defined in terms of the operations $+_{\N}$ and $\times_{\N}$ of the standard model $\N$ of PA.  

\medskip

While Tennenbaum's theorem is of interest to the \emph{semantic} approach to Nonstandard Analysis involving nonstandard models, $\RCAO$ is based on Nelson's \emph{syntactic} framework, and therefore Tennenbaum's theorem does not apply:  Any attempt at defining the (external) function `$+$ limited to the standard numbers' is an instance of \emph{illegal set formation}, forbidden in Nelson's \emph{internal} framework (\cite{wownelly}*{p.\ 1165}).  

\medskip

To be absolutely clear, lest we be misunderstood, Nelson's \emph{internal set theory} IST forbids the formation of \emph{external} sets $\{x\in A: \st(x)\}$ and functions `$f(x)$ limited to standard $x$'.  
Therefore, any appeal to Tennenbaum's theorem to claim the `non-computable' nature of $+$ and $\times$ from $\RCAO$ is blocked, for the simple reason that the functions `$+$ and $\times$ limited to the standard numbers' \emph{simply do not exist}.              
On a related note, we recall Nelson's dictum from \cite[p.\ 1166]{wownelly} as follows:
\begin{quote}
\emph{Every specific object of conventional mathematics is a standard set.} It remains unchanged in the new theory \textup{[IST]}.  
\end{quote}
In other words, the operations `$+$' and `$\times$', but equally so primitive recursion, in (subsystems of) IST, are \emph{exactly the same} familiar operations we  
know from (subsystems of) ZFC.  Since the latter is a first-order system, we however cannot exclude the presence of nonstandard objects, and internal set theory just makes this explicit, i.e.\ IST turns a supposed bug into a feature.    
\end{rem}

\section{The Explicit Mathematics theme around arithmetical comprehension}\label{revmathrcao}
In this section, we establish the EMT from Section~\ref{introke} for theorems $T$ such that $T^{*}\asa UT\asa (\exists^{2})$, i.e.\ at the level of arithmetical comprehension, the third Big Five system (See \cite{simpson2}*{III}) of Reverse Mathematics.   
\subsection{The EMT for weak K\"onig's lemma}\label{EMTWKL}
In this section, we establish the EMT for the weak K\"onig's lemma, the defining axiom of the Big Five system $\WKL_{0}$ (\cite{simpson2}*{I.10}).  
The \emph{uniform} version of weak K\"onig's lemma is defined as:
\be\tag{\textup{UWKL}}
(\exists \Phi^{1\di 1})(\forall T^{1}\leq_{1}1)\Big[(\forall n^{0})(\exists \beta^{1})(\overline{\beta}n\in T)\di (\forall x^{0})(\overline{\Phi(T)}\in T )\big], 
\ee
with set-theoretic notation rather than as in \cite{kohlenbach2}.  The nonstandard version is: 
\be\tag{\textup{WKL}$^{*}$}
(\forall^{\st} T^{1}\leq_{1}1)\Big[(\forall^{\st} n^{0})(\exists^{\st} \beta^{1})(\overline{\beta}n\in T)\di (\exists^{\st}\alpha^{1})(\forall m^{0})(\overline{\alpha}m\in T )\big].  
\ee
We are abusing notation by using `$T\leq_{1}1$' to denote that $T$ represents a binary tree.   
We will sometimes mention the notion of `infinite tree' and it will always be clear from context whether we mean the antecedent of UWKL or WKL$^{*}$.  
Furthermore, note the type mismatch between `infinite number' and `infinite tree'.  
\begin{thm}\label{trikke}
In $\RCAO$, we have $(\exists^{2})^{\st}\asa \paai \asa \WKL^{*}\asa \textup{E-UWKL}^{\st}$.
\end{thm}
\begin{proof}
By way of illustration of the use of $\Omega$-CA, we first prove $\paai\di (\exists^{2})^{\st}$.  We then prove
\be\label{tik}
\paai\di \WKL^{*}\di \textup{E-UWKL}^{\st} \di (\exists^{2})^{\st},
\ee
which establishes the theorem by Theorem \ref{markje} above.  

\medskip

First of all, assume $\paai$ and let $f^{1}$ be standard.  Clearly, $(\exists^{\st} x^{0})(f(x)=0)$ is equivalent to $(\exists x^{0}\leq M^{0})(f(x)=0)$, for any infinite $M^{0}$, even if $f$ involves additional standard parameters.   
Define standard $\psi^{(1\times 0)\di 0}$ as follows.
\be\label{treffend}
\psi(f,N^{0}):=
\begin{cases}
0 & (\exists x^{0}\leq N^{0})(f(x)=0)\\
1 & \textup{ otherwise}
\end{cases}.
\ee
By $\paai$, $\psi(f,M^{0})$ is $\Omega$-invariant, i.e.\ 
\[
(\forall^{\st}f^{1})(\forall N,M\in \Omega)(\psi(f,M)=_{0}\psi(f,N)).
\]
By $\Omega$-CA, there is standard $\varphi^{1\di 0}$ such that $(\forall^{\st}f^{1})(\forall M\in \Omega)(\psi(f,M^{0})=_{0}\varphi(f))$,  and $(\exists^{2})^{\st}$ immediately follows by the definition of $\psi$. 

\medskip

Secondly, the first implication in \eqref{tik} is immediate: By Theorem \ref{markje}, we have $(\exists^{2})^{\st}$ and hence $\WKL^{\st}$, as arithmetical comprehension implies weak K\"onig's lemma.  
Now, the consequent of $\WKL^{\st}$ is $(\exists^{\st}\alpha^{1})(\forall^{\st}n^{0})(\overline{\alpha}n\in T)$, and applying $\paai$ to the innermost $(\Pi_{1}^{0})^{\st}$-formula in the latter yields $\WKL^{*}$.       

\medskip

Thirdly, the final implication follows from \cite{kooltje}*{Proposition 3.4}.  Indeed, in the latter proof Kohlenbach establishes $ \UWKL\di (\exists^{2})$ by defining (primitive recursively) the functional $\varphi$ from $(\exists^{2})$ from the functional $\Phi$ from $\UWKL$;  The axiom of extensionality is only invoked for $\Phi$.  Furthermore, the axiom $\mathcal{T}_{st}$ from $\RCAO$ states that the Kleene recursor constant $R_{0}$ is standard, implying that all functionals defined by primitive recursion in $\RCAo$ are standard in $\RCAO$.  Hence, Kohlenbach's proof of $ \UWKL\di (\exists^{2})$ goes through relative to the standard world in $\RCAO$, i.e.\ we have the final implication in \eqref{tik}.  

\medskip

For the middle implication in \eqref{tik}, we first prove $(\exists^{2})^{\st}\di \UWKL^{\st}$ which will make what follows much clearer.  
To this end, for an infinite binary tree $T$, consider:
\be\label{eureukeu}
(\forall^{\st}n^{0})(\exists \alpha^{0})(|\alpha|=n\wedge 0*\alpha \in T) \vee (\forall^{\st}n^{0})(\exists \alpha^{0})(|\alpha|=n\wedge 1*\alpha \in T), 
\ee
which expresses that one of the branches originating from the root node of $T$ is infinite.  
As is clear from \eqref{eureukeu}, the notion of \emph{infinite branch} is a $(\Pi_{1}^{0})^{\st}$-formula.
To derive $\UWKL^{\st}$ from $(\exists^{2})^{\st}$, one starts at the root node of $T$ and determines if the 0-branch or 
the 1-branch is infinite using $(\exists^{2})^{\st}$, i.e.\ which disjunct of \eqref{eureukeu} holds.    
If the $n$-branch is chosen (for $n=0,1$), we define $\Phi(T)(0)$, the first element of the path in $T$, as the number  $n$.  Next, we move to the node $n$ and repeat the previous construction relative to $n$ to define $\Phi(T)(1)$, and so on.  

\medskip

Intuitively speaking, to prove the implication $\WKL^{*}\di \textup{E-UWKL}^{\st}$, we assume $\WKL^{*}$ and repeat the construction of $\Phi$ from the previous paragraph, but with both quantifiers `$(\forall^{\st}n^{0})$' in \eqref{eureukeu} replaced by $(\forall n^{0}\leq M)$ for infinite $M$.  
Because of $\WKL^{*}$, the resulting functional is $\Omega$-invariant and standard extensional.  We now spell out the details of this heuristic sketch.  

\medskip

Thus, assume WKL$^{*}$ and define the functional $\Psi(T, M)$ as follows:
We define $\Psi(f,M)(0)$ as $0$ if $(\forall m^{0}\leq M)(\exists \alpha^{0})(|\alpha|=m\wedge 0*\alpha \in T)$, and $1$ otherwise.  
For the general case, define 
\be\label{canona}
\Psi(f,M)(n+1):=
\begin{cases}
0 & (\forall m^{0}\leq M)(\exists \alpha^{0})(|\alpha|=m\wedge \Phi(T)(n)*\alpha \in T) \\
1 &\text{otherwise}
\end{cases}.
\ee
\medskip
Next, we  prove that $\Psi(T,M)$ is $\Omega$-invariant assuming $\WKL^{*}$.  
First of all, define the subtree $T_{\sigma}$ for $\sigma\in T$ by $(\forall \tau^{0}\leq_{0}1)(\tau\in T_{\sigma}\asa \sigma*\tau\in T)$.  
Next, fix infinite $M$ and let $T$ be some infinite binary tree. Now note that if $\Psi(T,M)(0)=0$, then the $0$-branch of $T$ is infinite, i.e.\ $(\forall^{\st} m^{0})(\exists \alpha^{0})(|\alpha|=m\wedge 0*\alpha \in T)$.  By definition, this implies that the tree $T_{0}$ is infinite, 
i.e.\ $(\forall^{\st} m^{0})(\exists \alpha^{0})(|\alpha|=m\wedge \alpha \in T_{0})$.  Now apply $\WKL^{*}$ \emph{to the infinite binary tree $T_{0}$} to obtain the existence of a standard binary sequence $\alpha$ 
such that $(\forall n^{0})(\overline{\alpha}n\in T_{0})$.  However, this implies by definition that $(\forall m^{0})(\exists \alpha^{0})(|\alpha|=m\wedge 0*\alpha \in T)$.       
By the definition of $\Psi$, we must have $\Psi(f,M)(0)=\Psi(f,N)(0)=0$ \emph{for any infinite} $N^{0}$.  
Similarly, one proves that $\Psi(T,M)(n)=\Psi(T,N)(n)$ for any finite $n$ and infinite $M,N$, and we obtain $\Psi(T,M)\approx_{1} \Psi(T,N)$ for infinite $N,M$ and $T$ any standard \emph{infinite} binary tree.    

\medskip

As $\Omega$-CA requires quantification over \emph{all} standard binary trees as in \eqref{tomega}, we need to specify the behaviour on finite trees.  
Thus, we define $\Theta(T,M)$ as  $\Psi(T,M)$ if $(\forall m^{0}\leq M)(\exists \alpha^{0})(|\alpha|=m\wedge \alpha \in T)$, and $0$ otherwise.  
Using $\WKL^{*}$ as in the previous paragraph, it is clear that $\Theta(T,M)$ is $\Omega$-invariant, i.e.\
\be\label{tomega}
(\forall^{\st}n^{0}, T^{1}\leq_{1})(\forall N,M\in \Omega)\big[ \Theta(T,M)\approx_{1}\Theta(T,N) \big]. 
\ee
Now let $\Phi$ be the `standard part' of $\Theta$ provided by $\Omega$-CA, i.e.\ we have 
$(\forall^{\st}T^{1}\leq 1)(\forall M\in \Omega)(\Theta(T,M)\approx_{1} \Phi(T) )$ and $\Phi$ is as required by UWKL$^{\st}$.  Now suppose $\alpha^{1}\in T$ is a standard path \emph{to the left of} $\Phi(T)$.  
By the definition of `to the left of', there is some standard $n_{0}$ such that $\overline{\alpha}n_{0}=\overline{\Phi(T)}n_{0}$ and $\alpha(n_{0}+1)<\Phi(T)(n_{0}+1)$.  However, then the tree $T_{\overline{\alpha}n_{0}}$ is infinite and applying $\WKL^{*}$ to this tree, we 
have $\Phi(T)(n_{0}+1)=\Theta(T, M)(n_{0}+1)=\alpha(n_{0}+1)$, a contradiction.  Thus, $\Phi(T,M)$ outputs the \emph{the left-most path in $T$}.  This immediately implies the standard extensionality of $\Phi$, i.e.\ we have
\[
(\forall^{\st}T^{1}, S^{1}\leq_{1}1)(T\approx_{1}S\di \Phi(T)\approx_{1} \Phi(S)),
\]  
and E-UWKL$^{\st}$ now follows from $\WKL^{*}$, and we are done.  
\end{proof}
The functional $\Psi(\cdot, M)$ from \eqref{canona} is called the \emph{canonical approximation} of the functional $\Phi$ from $\UWKL$, as it mirrors the way the latter functional is defined using $(\exists^{2})$.  
As discussed in Remark \ref{corkul}, the canonical approximation has rather low complexity, with possible applications to Hilbert's program for finitistic mathematics.  
Finally, in light of the above, the reader should be convinced that the equivalence involving $\WKL^{*}$ is note merely a `coding trick' (See also Theorem \ref{boon}).  

\medskip

The series of implications \eqref{tik} is useful as a template for establishing similar equivalences in a uniform way, as is clear from the proofs of the theorems in the next two sections.  We first prove the following corollaries.
\begin{cor}
In $\RCAO$, we have $(\exists^{2})\asa \paai \asa \WKL^{*}\asa \UWKL$.
\end{cor}
\begin{proof}
Immediate from \cite{kohlenbach2}*{Prop.\ 3.9} and \cite{bennosam}*{Cor.\ 14}.  
Alternatively and similar to the proof of the latter, it is possible to modify E-UWKL$^{\st}$ so that PF-TP$_{\forall}$ may be applied, yielding $\UWKL$.  
This involves dropping the `st' in the antecedent of UWKL$^{\st}$ and bringing the universal quantifier in the consequent to the front.   
\end{proof}
The principle E-UWKL$^{\st}$ may seem somewhat contrived, but actually follows directly from UWKL.  In particular, the standard extensionality in the former can be proved quite elegantly using PF-TP$_{\forall}$, as we establish now.  
\begin{thm}\label{boon}
In $\RCAO$, the implication $\UWKL \di \textup{E-UWKL}^{\st}$ can be proved `directly', i.e.\ without the use of the equivalences $\paai\asa (\exists^{2})\asa \UWKL$.
In particular, the functional from $\UWKL$ may be assumed to be standard extensional.   
\end{thm}
\begin{proof}
First of all, let $\Psi$ be the functional from UWKL and note that the axiom of extensionality for $\Psi$ can be brought into the form:
\[
(\forall S^{1}, R^{1}\leq_{1}1, k^{0})(\exists N^{0})(\overline{T}N=\overline{S}N\di \overline{\Psi(S)}k=\overline{\Psi(R)}k,
\]
by resolving both occurrences of `$=_{1}$' and bringing all quantifiers to the front.  
Now bring all the quantifiers in UWKL to the front, yielding:
\[
(\exists \Psi^{1\di 1})(\forall n^{0}, T^{1}\leq_{1}1)(\exists m^{0})\Big[ (\exists \alpha^{0})[|\alpha|=m \wedge\alpha\in T] \di \overline{\Psi(T)}n\in T    \Big].
\]
Using QF-AC$^{1,0}$, we obtain
\[
(\exists \Psi^{1\di 1}, \Xi^{2})(\forall n^{0}, T^{1}\leq_{1}1)\Big[ (\exists \alpha^{0})[|\alpha|=\Xi(T,n) \wedge\alpha\in T] \di \overline{\Psi(T)}n\in T    \Big],
\]
and adding the extensionality of $\Psi$, we obtain 
\begin{align*}
(\exists \Psi^{1\di 1}, \Xi^{2})&\Big[(\forall n^{0}, T^{1}\leq_{1}1)\big[ (\exists \alpha^{0})[|\alpha|=\Xi(T,n) \wedge\alpha\in T] \di \overline{\Psi(T)}n\in T    \big] \\
&\wedge (\forall S^{1}, R^{1}\leq_{1}1, k^{0})(\exists N^{0})(\overline{T}N=\overline{S}N\di \overline{\Psi(S)}k=\overline{\Psi(R)}k\Big].
\end{align*}
Applying QF-AC$^{1,0}$ to the second conjunct, we obtain
\begin{align}
(\exists \Psi^{1\di 1}, \Xi^{2}, \Phi^{2})&\Big[(\forall n, T^{1}\leq_{1}1)\big[ (\exists \alpha^{0})[|\alpha|=\Xi(T,n) \wedge\alpha\in T] \di \overline{\Psi(T)}n\in T    \big]\label{bonga} \\
&\wedge (\forall S^{1}, R^{1}\leq_{1}1, k^{0})(\overline{T}\Phi(S,R,k)=\overline{S}\Phi(S,R,k)\di \overline{\Psi(S)}k=\overline{\Psi(R)}k\Big].\notag
\end{align}
Applying PF-TP$_{\forall}$ to \eqref{bonga}, we may assume the functionals $\Psi$, $\Phi$, and $\Xi$ are standard.  
Now consider a standard binary tree $T'$ such that  $(\forall^{\st}k^{0})(\exists \alpha^{0})(|\alpha|=k \wedge \alpha\in T')$.  
For standard $n$, $\Xi(T', n)$ is standard and we have $(\exists \alpha^{0})[|\alpha|=\Xi(T',n) \wedge\alpha\in T']$, implying $\overline{\Psi(T')}n\in T' $ for any standard $n$, i.e.\ we have $\UWKL^{\st}$.  
Similarly, if $S\approx_{1}R$ for standard binary trees $S, R$, then $\Phi(S, T, k)$ is standard for standard $k$, implying $\overline{\Psi(S)}k=\overline{\Psi(R)}(k)$ for standard $k$.  
But the latter is just $\Psi(T)\approx_{1}\Psi(S)$ and $\Psi$ satisfies standard extensionality, implying E-UWKL$^{\st}$.   
\end{proof}
In light of the previous theorem, one can prove standard extensionality for any functional of type $1\di 1$ (or of similar typing) with a defining internal formula (without parameters).  
This is somewhat surprising in light of the discussion of \eqref{EXT}$^{\st}$ in \cite{brie}*{p.\ 1973}.   
We now prove a result like Theorem \ref{boon} for $(\exists^{2})$.  

\medskip

In \cite{bennosam}*{Cor.\ 14}, the equivalence $\paai \asa (\exists^{2})$ is proved indirectly using the operator $(\mu^{2})$ from \cite{bennosam, avi2, kohlenbach2}, defined as:
\be\tag{$\mu^{2}$}
(\exists \mu^{2})(\forall f^{1})\big[(\exists x^{0})f(x)=0 \di f(\mu(f))=0].
\ee
This axiom is equivalent to $(\exists^{2})$ over $\RCAo$ (See \cite{kooltje}) and $(\mu^{2})$ yields $\paai$ by applying PF-TP$_{\forall}$ to the former, i.e.\ $\mu$ is standard, and so is the witness $\mu(f)$ for standard $f^{1}$.  
We now show that $(\exists^{2})$ implies $\paai$ in the same way.  
\begin{cor}\label{boon2}
In $\RCAO$, the implication $(\exists^{2})\di \paai$ can be proved `directly', i.e.\ without the use of the equivalence $(\exists^{2})\asa (\mu^{2})$.  
In particular, the functional from $(\exists^{2})$ may be assumed to be standard extensional.     
\end{cor}
\begin{proof}
Similar to the proof of Theorem \ref{boon}, obtain via PF-TP$_{\forall}$ that $(\exists^{2})$ implies the standardness and standard extensionality of $\varphi$ as in $(\exists^{2})$.  
This amounts to applying the contraposition of PF-TP$_{\forall}$ to 
\begin{align*}
(\exists \varphi^{2}, \Xi^{2})&\Big[ (\forall f^{1})[(\exists x^{0})f(x)=0\asa \varphi(f)=0] \\
&\wedge (\forall g^{1}, h^{1})(\overline{g}\Xi(g,h)=\overline{h}\Xi(g,h)\di \varphi(g)=\varphi(h) ) \Big],
\end{align*}
which follows from $(\exists^{2})$, extensionality \eqref{EXT}, and QF-AC$^{1,0}$.  
Now assume $\paai$ is false, i.e.\ there is standard $h^{1}$ such that $(\forall^{\st}n)(h(n)=0)\wedge (\exists n_{0})h(n)\ne0$.  Define the standard sequences $f^{1}:=11\dots$ and $g^{1}$ as: 
\[
g(n):=
\begin{cases}
1 & (\forall k<n)h(k)=0 \\
0 & \textup{otherwise}
\end{cases}.
\]
Clearly $f\approx_{1} g$, implying $1=\varphi(f)=\varphi(g)=0$, by the definition of $\varphi$ and standard extensionality.  This contradiction implies that we must have $\paai$.
\end{proof}
A similar theorem can be proved for the Suslin functional and $\Paai$ from \cite{bennosam}*{\S4.2}.  
Finally, the following corollary suggests that Theorem \ref{boon} provides a nice template for classifying uniform principles.  
\begin{cor}\label{boon3}
In $\RCAO$, the implication $\UWKL\di \paai$ can be proved `directly', i.e.\ without the use of the equivalence $(\exists^{2})\asa \paai$.  
\end{cor}
\begin{proof}
Proceed as in the proof of the theorem and derive \eqref{bonga}, but without introducing $\Xi^{2}$, i.e.\ we have
\begin{align}
(\exists \Psi^{1\di 1}, \Phi^{2})&\Big[(\forall T \leq_{1}1)\big[ (\forall m^{0})(\exists \alpha^{0})[|\alpha|=m \wedge\alpha\in T] \di (\forall n^{0})\overline{\Psi(T)}n\in T    \big]\label{bonga2} \\
&\wedge (\forall S^{1}, R^{1}\leq_{1}1, k^{0})(\overline{T}\Phi(S,R,k)=\overline{S}\Phi(S,R,k)\di \overline{\Psi(S)}k=\overline{\Psi(R)}k\Big].\notag
\end{align}
By PF-TP$_{\forall}$, we may assume that the functionals $\Psi, \Phi$ are standard and that $\Psi$ is standard extensional as in the proof of the theorem.   
Now suppose $\paai$ is false, i.e.\ there is standard $h^{1}$ such that $(\forall^{\st}n)h(n)=0$ and $(\exists m)h(m)\ne 0$.
For $T'$ the (necessarily standard) full binary tree, define the standard binary tree $S'$ as (using $\Psi$ from \eqref{bonga2}):
\[
  \sigma \in S' \asa (\forall m<|\sigma|)h(m)=0 \vee (\forall i=0,1)(\Psi(T)(1)=i \di  (\forall j<|\sigma|)\sigma(j)=1-i )   \big ].
\]  
However, then $T'\approx_{1} S'$, implying $\Psi(T')\approx_{1}\Psi(S')$, but this contradicts \eqref{bonga2} as $S'$ is infinite (either 00\dots or 11\dots is completely in $S'$), but $\overline{\Psi(S')}n$ is not in $S'$ for large enough $n$, by the assumption on $h$.  
This contradiction yields $\paai$.  
\end{proof}  
Finally, we point out that extensionality \eqref{EXT} is essential in obtaining the above results.  
In particular, Kohlenbach has proved that \emph{without full extensionality}, the principle UWKL is not really stronger than WKL itself (\cite{kooltje}).    
  

\subsection{The EMT for $\Sigma_{1}^{0}$-separation}
Next, we establish the EMT for the  $\Sigma_{1}^{0}$-separation principle from \cite{simpson2}*{I.11.7}, known to be equivalent to $\WKL$;  The nonstandard version is:  
\begin{princ}[$\Sigma_{1}^{0}$-SEP$^{*}$]
For standard $f_{i}^{1}$ and $\varphi_{i}(n)\equiv (\exists n_{i})(f_{i}(n_{i},n)=0)$ such that $ (\forall^{\st}n)\neg[ \varphi_{1}^{\st}(n)\wedge \varphi_{2}^{\st}(n)] $, there is standard $Z^{1}$ such that
\be\label{starbuck}
(\forall n^{0})\big[  \varphi_{1}(n)  \di n\not\in Z \wedge \varphi_{2}(n)\di n\in Z \big].
\ee
\end{princ}
The principle $\Sigma_{1}^{0}$-SEP$^{*}$ states the existence of a separating set for $(\Sigma_{1}^{0})^{\st}$-formulas, but for \emph{all} numbers, not just the standard ones.  
Similarly, $\textup{U}\Sigma_{1}^{0}\textup{-SEP}$ is: 
\begin{princ}[U$\Sigma_{1}^{0}$-SEP] \label{SEP11}
For $\varphi_{i} (n, f)\equiv(\exists n_{i})(f(n_{i}, n)=0)$, we have
\begin{align}
\big(\exists & F^{(1\times 1\times 0)\di 0}\big)(\forall f^{1},g^{1})\Big[ (\forall n)\neg[ \varphi_{1} (n,f)\wedge \varphi_{2} (n,g)] \di \notag \\ 
& (\forall n^{0})[\varphi_{1} (n,f) \di F(f,g,n)=1  ]\wedge (\forall n) [\varphi_{2} (n,g)\di F(f,g,n)=0 ] \Big].\label{jesestiasep}
\end{align}
\end{princ}

\begin{thm}\label{doooook}
In $\RCAO$, $\paai\asa \Sigma_{1}^{0}\textup{-SEP}^{*}\asa \textup{E-U}\Sigma_{1}^{0}\textup{-SEP}^{\st}\asa (\exists^{2})^{\st}$.  
\end{thm}
\begin{proof}
Clearly, we could exploit the equivalence between $\Sigma_{1}^{0}$-separation and weak K\"onig's lemma (\cite{simpson2}*{IV.4.4}), but we instead prove the following implications: 
\be\label{gen}
\paai \di \Sigma_{1}^{0}\textup{-SEP}^{*}\di \textup{E-U}\Sigma_{1}^{0}\textup{-SEP}^{\st} \di (\exists^{2})^{\st}, 
\ee
which establishes the theorem.  
The first implication in \eqref{gen} is trivial as \eqref{starbuck} is a $\Pi_{1}^{0}$-formula, and the final implication follows from \cite{yamayamaharehare}*{Theorem 3.6}.  
One can also use the observation that the final part of the proof of \cite{simpson2}*{IV.4.4} implies that $\textup{U}\Sigma_{1}^{0}\textup{-SEP}\di \UWKL$, and the same for 
the standard extensional versions.    

\medskip
   
For the second implication, note that $ \Sigma_{1}^{0}\textup{-SEP}^{*}$ implies for standard $f_{i}^{1}$ that:
\be\label{test}
 (\forall^{\st}n)[ \neg\varphi_{1}^{\st}(n, f_{1})\vee \neg\varphi_{2}^{\st}(n, f_{2})]\di  (\forall n)[\neg \varphi_{1}^{}(n, f_{1})\vee \neg\varphi_{2}^{}(n, f_{2})].
\ee
Now let $h_{i}^{1}$ be such that $h_{i}(k,M)=0\asa (\forall n_{i}^{0}\leq M)f_{i}(n_{i}, k)\ne0$, and note that by \eqref{test}, we have $h_{1}(k,M)=0 \vee h_{2}(k,M)=0$ for standard $k$ and $M\in \Omega$.  
Next, define the functional $\Psi$ as follows:
\be\label{dolpi}
\Psi(f_{1},f_{2},M)(n):=
\begin{cases}
~0 & h_{1}(n,M)=0 \wedge h_{2}(n,M)\ne 0 \\
~1 & h_{1}(n,M)\ne0 \wedge h_{2}(n,M)= 0 \\
~2 & h_{1}(n,M)\ne0 \wedge h_{2}(n,M)\ne0 \\
~3 & h_{1}(n,M)=0 \wedge h_{2}(n,M)= 0 \\
\end{cases}
\ee
We now show that $\Psi$ is both as required for $\textup{U}\Sigma_{1}^{0}\textup{-SEP}^{\st}$ and $\Omega$-invariant, in case the standard $f_{i}$ satsify $(\forall^{\st}n)[ \neg\varphi_{1}^{\st}(n, f_{1})\vee \neg\varphi_{2}^{\st}(n, f_{2})]$.  

\medskip

Indeed, for standard $n$, if $\varphi_{1}^{\st}(n, f_{1})$ then by assumption and \eqref{test}, we have $\neg\varphi_{2}(n, f_{2})$, and the second case in \eqref{dolpi} holds (for any infinite $M$). 
If $\varphi_{2}^{\st}(n, f_{2})$ holds for standard $n$, then similarly $\neg\varphi_{1}(n, f_{1})$ by the previous, and the first case in \eqref{dolpi} holds (for any infinite $M$).  
Since $\varphi_{2}^{\st}(n, f_{2})\wedge \varphi_{1}^{\st}(n, f_{1})$ is impossible by assumption, the third case in \eqref{dolpi} does not occur.  If for some standard $n_{0}$, we have the fourth case (or even only $\neg\varphi_{2}^{\st}(n_{0}, f_{2})\wedge \neg\varphi_{1}^{\st}(n_{0}, f_{1})$), consider the 
functions $f_{3}, f_{4}$ defined for any $k$ as $f_{3}(k,n)=f_{1}(k,n)$ and $f_{4}(k,n)=f_{2}(k,n)$ for $n\ne n_{0}$ and $f_{3}(k,n_{0})=f_{4}(k,n_{0})=f_{1}(k,n_{0})\times f_{2}(k,n_{0})$.  
By assumption, we have $(\forall^{\st}n)[ \neg\varphi_{3}^{\st}(n, f_{3})\vee \neg\varphi_{4}^{\st}(n, f_{4})$ and \eqref{test} 
applied to the latter yields $\neg\varphi_{1}(n_{0}, f_{1})\wedge \neg\varphi_{2}(n_{0}, f_{2})$.    
Hence, if the fourth case in \eqref{dolpi} occurs, it does so for every $M\in \Omega$.  

\medskip  

As $\Omega$-CA requires quantification over \emph{all} standard sequences $f_{i}^{1}$ as in \eqref{tomega2}, we need to specify the behaviour when the separation assumption $(\forall^{\st}n)[ \neg\varphi_{1}^{\st}(n, f_{1})\vee \neg\varphi_{2}^{\st}(n, f_{2})]$ is not met.      
Thus, we define $\Theta(f, g, M)$ as  $\Psi(f, g,M)$ if $(\forall n\leq M)\big[ (\forall n_{1}\leq M)f(n_{1},n)\vee (\forall n_{2}\leq M)g(n_{2}, n) \big]$, and $0$ otherwise.  
Using $ \Sigma_{1}^{0}\textup{-SEP}^{*}$ as in the previous paragraph, it is clear that $\Theta(T,M)$ is $\Omega$-invariant, i.e.\ we have
\be\label{tomega2}
(\forall^{\st}  f^{1},g^{1})(\forall N,M\in \Omega)\big[ \Theta(f,g,M)\approx_{1}\Theta(f,g,N) \big]. 
\ee  
The axiom $\Omega$-CA now provides a standard functional $\Phi(\cdot)\approx_{1}\Theta(\cdot, M)$ which satisfies $\textup{U}\Sigma_{1}^{0}\textup{-SEP}^{\st}$ by the above.  
As to the standard extensionality of $\Phi$, note that if $\varphi_{i}^{\st}(n, f_{i})$, i.e.\ in one the first two cases of \eqref{dolpi}, this extensionality property is immediate due to \eqref{test}.  By the latter, the third case of \eqref{dolpi} also does not occur for standard $h_{1}, h_{2}$ such that $h_{i}\approx_{1}f_{i}$.  For the final case in \eqref{dolpi}, a similar argument involving $f_{3}, f_{4}$ from the previous paragraph yields standard extensionality.    
\end{proof}
We could prove a version of Theorem \ref{boon} (and corollaries) for $\Sigma_{1}^{0}$-separation.  
\begin{cor}
In $\RCAO$, $\paai\asa \Sigma_{1}^{0}\textup{-SEP}^{*}\asa \textup{U}\Sigma_{1}^{0}\textup{-SEP}\asa (\exists^{2})$.  
\end{cor}
\begin{proof}
Immediate from \cite{bennosam}*{Cor. 12} and \cite{yamayamaharehare}*{Theorem 3.6}.  
\end{proof}
\subsection{The EMT for the intermediate value theorem}
Next, we establish the EMT for the intermediate value theorem (IVT).  Although the latter is provable in $\RCA_{0}$ (\cite{simpson2}*{II.6.6}), the uniform version of UIVT is equivalent to $(\exists^{2})$, as discussed by Kohlenbach in \cite{kohlenbach2}*{p.\ 293} and in Remark \ref{difconc}.    

\medskip

With regard to notation, let `$f\in C[0,1]$' mean that $f$ is $\eps$-$\delta$-continuous on $[0,1]$ and `$f\in \overline{C}[0,1]$' that $f\in C[0,1]$ and $f(0)\geq 0$ and $f(1)\leq 0$.  Appending of `st' to $C$ and $\overline{C}$ means that all quantifiers are relative to `st'.  
As explained in Remark~\ref{takkap}, the exact choice of continuity (involving $C$ or $C^{\st}$) is immaterial.  
The uniform and nonstandard versions of IVT are:  
\be\tag{$\IVT^{*}$}
(\forall^{\st}f^{1\di 1}\in \overline{C}^{\st}[0,1])(\exists^{\st}x^{1}\in [0,1])(f(x)=0).
\ee
\be\tag{$\UIVT$}
(\exists \Phi^{(1\di 1)\di 1})(\forall f^{1\di 1} \in \overline{C}[0,1])(f(\Phi(f))=0).
\ee
Note that $\IVT^{\st}$ is weaker than $\IVT^{*}$, as the latter (and also UIVT$^{\st}$) involves `$\approx$' rather than `$=$', which turns out to yield quite a difference in strength.
\begin{thm}\label{2trikke}
In $\RCAO$, we have $(\exists^{2})^{\st}\asa \paai \asa \IVT^{*}\asa \textup{E-UIVT}^{\st}$.
\end{thm}
\begin{proof} 
To establish the equivalences in the theorem, we prove
\be\label{ht}
\paai \di \textup{IVT}^{*}\di  \textup{E-UIVT}^{\st}\di (\exists^{2})^{\st}.
\ee  
which establishes the theorem by Theorem \ref{markje} above.  
The first implication in \eqref{ht} is trivial as $\IVT$ is provable in $\RCA_{0}$ (See e.g.\ \cite{simpson2}*{II.6.6}) and $z^{1}=0$ is a $\Pi_{1}^{0}$-formula for reals $z$. 
The final implication follows from the proofs of \cite{kohlenbach2}*{Proposition 3.14} and \cite{kooltje}*{Prop.\ 3.9} in the same way the final implication in \eqref{tik} is proved.  

\medskip

For the remaining implication in \eqref{ht}, 
it should first be noted that in the proof that \cite{kohlenbach2}*{Proposition 3.14.3} implies $(\exists^{2})$, the intermediate value functional $F$ is only applied to polynomials $f_{y}(x):=yx-y$ and $f_{y}(x):=yx$ to obtain a discontinuous function, and hence $(\exists^{2})$.  
Hence, it suffices to obtain $\textup{E-UIVT}^{\st}$ limited to standard polynomials (which can be coded by standard reals).  
This situation is reminiscent of the `usual' Brouwerian counterexample involving IVT where only very special functions are used (See e.g.\ \cite{bridges1}*{p.\ 4} or \cite{mandje2}*{p.\ 11}).   

\medskip

We now prove the remaining implication in \eqref{ht}.  To this end, we prove that $\IVT^{*}$ implies $ \textup{E-UIVT}^{\st}$ limited to standard polynomials.  
To quantify over the latter, we will use the notation $(\forall^{\st}f^{1\di 1}\in P)$, although polynomials may be coded by reals.      
We also write $(\forall^{\st}f^{1\di 1}\in \overline{P})$ to mean that $f(0)>0 $ and $f(1)<0$.  
Similar to the proof of Theorem~\ref{tik}, we will mimic the proof of $(\exists^{2})\di \UIVT$ involving the usual `interval-halving' technique.  

\medskip

First of all, following \cite{simpson2}*{II.6.6-7}, we have the `approximate' IVT: 
\[\textstyle
(\forall^{\st}f^{1\di1}\in \overline{P})(\forall k^{0})(\exists x^{1}\in (0,1))(|f(x)|< \frac{1}{k}).  
\]
Now, by the continuity of $f$, we can replace `$(\exists x^{1}\in (0,1))(|f(x)|<\frac{1}{k})$' by a $\Sigma_{1}^{0}$-formula, 
say $(\exists n^{0})\varphi(f, x, n)$, with $\varphi$ quantifier-free.    
Furthermore, it is easy to find a (primitive recursive) witnessing function for this existential quantifier (again using the continuity of $f$).  
%
%
Actually, this witnessing function is nothing more than a realizer for the constructive `approximate' version of IVT (See e.g.\ \cite{bish1}*{Theorem~4.8, p.\ 40}).  
Hence, we can treat $(\forall k^{0})(\exists x^{1}\in (0,1))(|f(x)|< \frac{1}{k})$ as a $\Pi_{1}^{0}$-formula, say $(\forall k^{0})\psi(f, k, 0, 1)$, with $\psi$ quantifier-free and the two last variable places for the interval end points in the former formula.         

\medskip

Now define $\Psi(f, M)(0)$ as $0$ if $(\forall k^{0}\leq M)\psi(f, k, 0,\frac{1}{2})$, and $\frac{1}{2}$ otherwise.  In general, the functional $\Psi$ is defined as:
\[
\Psi(f,M)(n+1):=
\begin{cases}
\Psi(f, M)(n) &  (\forall k^{0}\leq M)\psi(f, k, \Psi(f, M)(n),\Psi(f, M)(n)+\frac{1}{2^{n+1}}) \\
\Psi(f, M)(n)+\frac{1}{2^{n+1}} & \text{otherwise}
\end{cases}.
\]
We now prove that $\Psi(f,M)$ is $\Omega$-invariant for standard $f\in \overline{P}$ and $M\in \Omega$.  
If $\Psi(f, M)(0)=0$, then $f$ becomes arbitrarily small on $[0,\frac12]$ \emph{relative to `\st'}, i.e.\ we have that $(\forall^{\st} k^{0})(\exists^{\st} x^{1}\in (0,\frac{1}{2}))(|f(x)|< \frac{1}{k})$.   
Since $f$ is a standard polynomial, this implies $(\exists^{\st}x^{1}\in [0,\frac{1}{2}])f(x_{0})\approx 0$.  Applying IVT$^{*}$ for the interval $[0,x_{0}]$, we obtain 
$(\exists^{\st}x^{1}\in [0,\frac{1}{2}])f(x)= 0$.  But then $f$ becomes arbitrarily small on $[0,\frac{1}{2}]$, i.e.\ $(\forall  k^{0})(\exists  x^{1}\in (0,\frac{1}{2}))(|f(x)|< \frac{1}{k})$, and we have 
$\Psi(f, M)(0)=\Psi(f,N)(0)=0$ for any $N\in \Omega$ by definition.  

\medskip

Similarly, one proves that $\Psi(f,M)(n)=\Psi(f,N)(n)$ for any finite $n$ and infinite $M,N$, 
and we obtain $\Psi(f,M)\approx_{1} \Psi(f,N)$ for infinite $N,M$ and $f$ any standard polynomial.  
Furthermore, it is easy to see that $\Psi(f,M)$ provides the \emph{left-most} intermediate value of $f$.  
Applying $\Omega$-CA now yields E-UIVT$^{\st}$ limited to standard polynomials, and we are done.          
\end{proof}
\begin{cor}
In $\RCAO$, we have $(\exists^{2})\asa \paai \asa \IVT^{*}\asa \UIVT$.
\end{cor}
\begin{proof}
Immediate from the theorem, \cite{kohlenbach2}*{Prop.\ 3.14}, and \cite{bennosam}*{Cor.\ 14}.  
We also sketch a `direct' proof of $\UIVT\di \paai$ below in Remark \ref{takkap}.  
\end{proof}
We finish this section with a remark on continuity.
\begin{rem}[Continuity]\label{takkap}\rm
First of all, as is clear from \cite{kohlenbach2}*{Prop.\ 3.14} and as noted in the previous proof, UIVT limited to various continuity classes is still equivalent to $(\exists^{2})$.  
Similarly, one can replace $\overline{C}^{\st}$ in IVT$^{*}$ and $\textup{E-UIVT}^{\st}$ by $\overline{C}$ and the proof of $\IVT^{*}\di \textup{E-UIVT}^{\st}$ still goes through.  
Indeed, with $\Psi$ defined as in the proof of the theorem, one can prove $(\forall^{\st}f^{1\di 1})(\forall N,M\in \Omega)\big[f\in C\di \Psi(f,N)\approx_{1} \Psi(f,M) \big]$ and apply Corollary~\ref{genalli}.   
To prove the former, it does seem $\WKL^{\st}$ is needed (which follows from E-UIVT$^{\st}$ limited to standard polynomials).  
Since internal formulas are part of the original language (of $\RCAo$ or $\RCA_{0}$), $C$-functions are arguably more interesting objects of study than $C^{\st}$-functions (from the point of view of the EMT).   
One could also work with representations $\Phi^{1\di1}\in \mathfrak{C}[0,1]$, where the latter denotes the definition of continuity on Cantor space, i.e.\
\be\label{wantonuity}
(\forall \alpha\leq_{1}1, k^{0})(\exists N^{0})(\forall \beta\leq_{1}1)(\overline{\alpha}N=\overline{\beta}N\di \Phi(\alpha)(k)=\Phi(\beta)(k)), 
\ee
and $\Phi$ is extensional with regard to real equality as in Remark \ref{keepintireal}.  
Secondly, we can prove a version of Theorem \ref{boon} for the intermediate value theorem by replacing $f^{1\di 1}\in C[0,1]$ by the definition of 
continuity from Reverse Mathematics (\cite{simpson2}*{II.6.1}) involving a so-called associate \emph{of type 1}.  Indeed, by \cite{kohlenbach4}*{Prop.\ 4.10 and Prop.\ 4.4}, every pointwise continuous 
function satisfies \cite{simpson2}*{II.6.1} given WKL, and the latter already follows from UIVT limited to polynomials.  

\medskip

For instance, to prove that the aforementioned limited version of UIVT implies $\paai$, assume $h^{1}$ does not satisfy the latter, let $f$ be the
function from the usual Brouwerian counterexample to IVT (See \cite{mandje2}*{p.\ 11}) and if $\Phi(f)>\frac{1}{2}$ define $g$ by $g(x):=f(x)-\sum_{k=0}^{\infty} k(i) \frac{x^{i}}{i!}$, 
where $k(i)=1\asa (\exists n\leq i)h(n)\ne 0$.  Abusing notation somewhat, we have $f\approx_{1}g$, implying $\Phi(f)\approx_{1} \Phi(g)$, but this yields a contradiction as $\Phi(g)$ must satisfy $\Phi(g)\leq \frac{1}{3}$.  For the 
case $\Phi(f)\leq \frac{1}{2}$, define $g(x):= f(x)+\dots$, and we obtain $\UIVT\di \paai$.   
\end{rem}
\subsection{The EMT for the Weierstra\ss~extremum theorem}
Finally, we establish the EMT for the Weierstra\ss~maximum theorem, equivalent to WKL by \cite{simpson2}*{IV.2.3}.  
\be\tag{\textup{UWEIMAX}}
(\exists  \Phi^{(1\di 1)\di 1})(\forall f \in C[0,1])(\forall y\in [0,1])(f(y)\leq f(\Phi(f))).
\ee
\be\tag{\WEIMAX$^{*}$}
(\forall^{\st} f \in {C}[0,1])(\exists^{\st} x^{1}\in [0,1])(\forall y^{1}\in [0,1])(f(y)\leq f(x)).
\ee
By \cite{kohlenbach2}*{Prop.\ 3.14}, the exact choice of continuity in UWEIMAX does not matter.     
Note that $f\in {C}[0,1]$ does not involve `st' in WEIMAX$^{*}$.  
\begin{thm}
In $\RCAO$, $\paai \asa \textup{WEIMAX}^{*}\asa  \textup{UWEIMAX}\asa (\exists^{2})$.   
\end{thm}
\begin{proof}
To establish the theorem, we now prove:
\be\label{htto}
\paai \di \textup{WEIMAX}^{*}\di  \textup{E-UWEIMAX}^{\st}\di (\exists^{2})^{\st}.
\ee  
First of all, the final implication again follows from \cite{kohlenbach2}*{Proposition 3.14}.  Furthermore, to obtain UIVT from UWEIMAX, apply the latter to $-|f|$ (for $f\in \overline{C}$) and note that the maximum of $-|f|$ must be an intermediate value of $f$.  However, as noted in the proof of Theorem \ref{2trikke}, UIVT is only applied to polynomials to obtain $(\exists^{2})$ in the proof of \cite{kohlenbach2}*{Proposition 3.14}.  Hence, UWEIMAX and WEIMAX$^{*}$ may also be limited to certain elementary functions. 

\medskip

Secondly, we prove the first implication in \eqref{htto}.  Using \cite{kohlenbach2}*{Proposition 3.14} and \cite{bennosam}*{Cor.\ 15}, $\paai$ is equivalent to UWEIMAX, and by PF-TP$_{\forall}$ the functional in the latter is standard, immediately implying WEIMAX$^{*}$.  
We also list more conceptual `direct' proofs for Lipschitz continuous functions (with factor one) and for $\Phi\in \mathfrak{C}[0,1]$.  In light of the results in \cite{polarhirst}, it should be straightforward to convert $f\in C[0,1]$ into $\Phi\in \mathfrak{C}[0,1]$ using $(\exists^{2})$.   
Note that the first implication in \eqref{htto} is \emph{not} trivial, as the innermost universal formula of the Weierstra\ss~maximum theorem 
is $\Pi_{1}^{1}$.  Nonetheless, this formula is \emph{equivalent} to a $\Pi_{1}^{0}$-formula, namely if we restrict the real quantifier to the rationals. 
Thus, assume $\paai$ and consider standard $\Phi\in \mathfrak{C}[0,1]$.  We first prove that $\Phi\in \mathfrak{C}^{\st}[0,1]$, i.e.\ that $\Phi$ also is continuous relative to the standard world as in \eqref{wantonuity}$^{\st}$.  

\medskip

To this end, consider the proof of \cite{kohlenbach4}*{Prop.\ 4.10} in which it is proved that a functional $\Phi^{1\di 1}\in \mathfrak{C}[0,1]$ has a modulus of uniform continuity, assuming $\WKL$.  
In the latter proof, Kohlenbach defines a sequence of infinite binary trees $T_{k,n}$ using (only) $\Phi$, and the sequence of paths through these trees (the sequence exists via WKL) is used to define the characteristic function of the formula in square brackets in the following formula:
\be\label{leiter}
(\forall k, f^{1}\leq_{1}1)(\exists N)\big[(\forall h^{1}, g^{1}\leq_{1}1)(\overline{f}N=\overline{g}N=\overline{h}N  \di \Phi(h)(k)=\Phi(g)(k)) \big]
\ee
Now if $\Phi$ is additionally standard, the tree $T_{k,n}$ will also be standard, and $\UWKL$ (available via Theorem \ref{trikke}) and PF-TP$_{\forall}$ yield a standard sequence of paths through (all of) $T_{k,n}$.  It is then easy to show that the formula in square brackets in \eqref{leiter} now has a \emph{standard} characteristic function.  
However, then $\paai$ yields that for standard $k$ and standard $f^{1}\leq_{1}1$
\[
(\exists^{\st} N^{0})\big[(\forall h^{1},g^{1}\leq_{1}1)(\overline{f}N=\overline{g}N =\overline{h}N \di \Phi(f)(k)=\Phi(g)(k)) \big],
\]     
as the formula in square brackets may be treated as quantifier-free with standard parameters.  With some coding, $\Phi\in \mathfrak{C}^{\st}[0,1]$ now follows, i.e.\ we have \eqref{leiter}$^{\st}$.  

\medskip

Thirdly, since $\paai\di (\exists^{2})^{\st}\di \WKL^{\st}$, \cite{simpson2}*{IV.2.3} implies the Weierstra\ss~maximum theorem relative to `st'.  By \cite{kohlenbach4}*{Prop.\ 4.4 and 4.10}, we may apply the maximum theorem for standard $\Phi\in \mathfrak{C}[0,1]$ (which satisfy $\Phi\in \mathfrak{C}^{\st}[0,1]$ by the previous paragraph). Hence, there is standard $x_{0}\in [0,1]$ such that $(\forall q^{0}\in [0,1])(\Phi(q)\leq \Phi(x_{0}))$, where we used $\paai$ to obtain the final $\Pi_{1}^{0}$-formula.    
However, for any $\Phi\in \mathfrak{C}[0,1]$ and $x\in [0,1]$, we have
\be\label{waldord}
 (\forall q^{0}\in [0,1])(\Phi(q)\leq \Phi(x) )\asa (\forall  y^{1}\in [0,1])(\Phi(y)\leq \Phi(x)).  
\ee   
The reverse implication of \eqref{waldord} is trivial, while the forward one is a simple application of $\Phi\in \mathfrak{C}[0,1]$.  
Hence, \eqref{waldord} implies WEIMAX$^{*}$ limited to $\Phi\in \mathfrak{C}[0,1]$.  For Lipschitz continuous functions with factor one, note that such functions are by definition also Lipschitz continuous relative to `st'.     
Therefore, the same proof involving \eqref{waldord} yields WEIMAX$^{*}$ for such functions.  

\medskip

Finally, for the remaining implication in \eqref{htto}, we only provide a sketch, as the fomer is proved in much the same way as the middle implication in \eqref{ht}.  
Indeed, consider the following formula $(\forall k^{0} )(\exists x^{1}\in [0,1])(|\sup_{y\in [0,1]}f(y)-f(x)|<\frac{1}{k})$.  
Note that $\WEIMAX^{*}$ implies the existence of the supremum by \cite{simpson2}*{IV.2.3}.  As in the proof of Theorem \ref{2trikke}, there is 
a witnessing function for the existential quantifiers in the previous formula;  The functional $\Psi(f,M)$ is then built in the same way as in the proof of Theorem \ref{2trikke}, and we are done.

\medskip

In conclusion, we note that a slick proof of this theorem proceeds by proving \eqref{htto} for standard Lipschitz continuous functions and then proving `full' \eqref{htto} using the already established equivalences and (the proof of) \cite{kohlenbach2}*{Prop.\ 3.14}.   
\end{proof}
%
%
The above proofs reveal a template of the form which may be applied to obtain the EMT for RT$(1)$ and \cite{simpson2}*{I.10.3.9}, using \cite{yamayamaharehare}*{Theorems 4.2 and 4.3}.  
In Remark~\ref{difconc} below, we elaborate on this template.  Furthermore, as in Remark~\ref{takkap}, a version of Theorem \ref{boon} may be proved for the Weistra\ss~maximum theorem by restricting continuity 
to the usual definition in Reverse Mathematics.    
  
\subsection{Concluding remarks}\label{conrem}
We finish this section with some concluding remarks.
\begin{rem}[Hilbert's program]\label{corkul}\rm
We discuss the connection of the above results to Hilbert's program for finitist mathematics. 
We are motivated by Tait's analysis (\cite{tait1}) that the formal system PRA captures Hilbert's notion of finitist mathematics, 
and Burgess' detailed study of how Reverse Mathematics contributes to Hilbert's program (\cite{burgess}).  
The following quote is essential:  
\begin{quote}
[\dots] whereas a finitist cannot know that everything provable in PRA is finitistically provable, a finitist can know that everything provable in a bounded fragment of PRA such as EFA is finitistically provable.  
This positive fact is the other side of the coin from the negative fact that bounded fragments do not exhaust finitistic provability as (according to the Tait analysis) PRA provability does. (\cite{burgess}*{p.\ 139})
\end{quote}
In short, to establish a partial realization of Hilbert's program, it is essential according to Burgess that results are ultimately provable in a \emph{subsystem} of PRA, like EFA ($=I\Delta_{0}+\EXP$).  
Now, the functional $\Psi(T,M)$ from \eqref{canona} is definable in the $\Pi_{2}^{0}$-conservative extension of EFA from \cite{bennosam}*{Cor.\ 8}, i.e.\ the latter nonstandard system has a PRA-consistency proof.  
Thus, after following the latter proof, a finitist can conclude that the latter functional is unproblematic finitistically.  

\medskip

Furthermore, one can prove in the same EFA-based system that every standard functional $\Xi^{1\di 1}$ which outputs the left-most path $\Xi(T)$ in the standard binary tree $T$ satisfies $\Xi(T)\approx_{1}\Psi(T,M)$, and vice versa.  In other words, a finitist can accept the correctness of the hypothetical statement `If a functional as in UWKL exists, then it equals a finitistically acceptable object'.  It should be noted that a similar argument works for the fan functional \eqref{MUC} from Section \ref{fannypack2}, which happens to be inconsistent with UWKL.           
\end{rem}
As to intuitionistic mathematics, we remark the following.
\begin{rem}\rm
For L.E.J.\ Brouwer, the real numbers $\R$ constituted a `unsplittable continuum', exemplified by Brouwer's rejection of $x>_{\R}0 \vee x\leq_{\R}0$, a special case of \emph{tertium non datur}.  
A similar observation regarding the `syrupy' continuum in intuitionistic mathematics is made by van Dalen in \cite{dalencont}.  
The results in the previous theorems and corollaries go the opposite way:  In our system, the maximum or intermediate values of continuous functions are determined by discrete case distinctions 
as done in the canonical functional.  In this way, a `very discrete' picture of the continuum emerges.  Furthermore, in our opinion, the canonical approximations endow the original functionals with plenty of `numerical meaning', though not 
the kind envisaged by Brouwer and other constructivists.    
In Section \ref{fannypack2}, we establish the EMT for principles from intuitionistic mathematics.   
\end{rem}
Next, we discuss a connection to intuitionistic logic due to Kohlenbach.  
\begin{rem}\rm\label{difconc}
As noted above, while the intermediate value theorem is provable in $\RCA_{0}$, the uniform version is equivalent to $(\exists^{2})$.  
Similarly, the statement SUP that every continuous function has a supremum is equivalent to the Weierstra\ss~maximum theorem \cite{simpson2}*{IV.2.3}, 
but the uniform version of SUP is much weaker than $(\exists^{2})$ (See \cite{kohlenbach2}*{\S3} and Corollary \ref{wusup}).  This behaviour can be explained as follows.   

\medskip

Following Kohlenbach (\cite{kohlenbach2}), the cause of the difference in behaviour between the maximum theorem and SUP, is that the latter can be proved from the fan theorem in intuitionistic logic, while the former by contrast requires classical logic.  This use of classical logic results in a discontinuity at the uniform level and hence $(\exists^{2})$ due to so-called Grilliot's trick (See \cite{kohlenbach2}*{\S3} and \cite{grilling}).  
%
%
This leads us to the following conjecture, where `BISH' is Errett Bishop's \emph{Constructive Analysis} (\cite{bish1}).  
\begin{conj}\label{braddd}~\rm
For a theorem $T$ provable in $\ACA_{0}$, there are two categories:
\begin{enumerate}  
\item~\big[\textup{BISH} $\vdash$ ({$T\di \WKL$})\big] We have $\paai\asa (\exists^{2}) \asa T^{*}\asa UT$.
\item~\big[\textup{BISH} $\vdash$ ({$\FAN\di T$~})\big] We have~~~~~ $T^{\st}\asa T^{*}\asa UT^{\st}$.  
\end{enumerate}  
\end{conj} 
Examples of the second case of the conjecture are discussed in Section \ref{fannypack}.  In particular, we study the fan theorem itself, the Heine-Borel lemma, Riemann integration, and the supremum of continuous functions.        
\end{rem}
Following this conjecture, the following theorems should fall into the first category: \emph{Peano's theorem for $y'=f(x,y)$, binary expansion of reals, Jordan matrix decomposition, Ramsey's theorem \textup{RT}$(1)$, contraposition of Heine-Borel compactness, G\"odel's completeness theorems, Brouwer's fixed point theorem, the Hahn-Banach theorem, Weierstra\ss~approximation theorem, the Hilbert and Robson basis theorems \(\cite{baasje}\), \WWKL, etc}.

\medskip

Examples of theorems which should fall in the second category: \emph{contraposition of $\WWKL$ and $\Sigma_{1}^{0}$-separation, Riemann integration of continuous functions, Heine-Borel compactness, theorems from the previous category involving \textbf{unique} existence \(\cite{ishberg}\), existence of supremum for $f\in C[0,1]$, etc.} 

\medskip

Finally, we discuss our choice of framework.  
\begin{rem}\label{flack}\rm
As a consequence of the above results, we observe that the functional $\Phi$ from $\UWKL$ (which may be assumed to output the left-most path) equals the functional $\Theta(\cdot, M)$ from \eqref{tomega} for infinite $M$ and \emph{standard input}.   
Similarly, the functional $\varphi$ from $(\exists^{2})$ equals $\psi(\cdot, M)$ from \eqref{treffend} for infinite $M$ and \emph{standard input}.  The apparent restriction to \emph{standard input} is only a limitation of our choice of framework: Indeed, in \emph{stratified} Nonstandard Analysis, the unary predicate `$\st(x)$' is replaced by the binary predicate `$x\rel y$', to be read `$x$ is standard relative to $y$' (\cites{hrbacek3, hrbacek4, hrbacek5, aveirohrbacek, peraire}).  In this framework, we could prove the following:  
\[
(\forall f^{1})(\forall M\sler f)[\psi(f, M)=_{0} \varphi(T)]\wedge (\forall T^{1}\leq_{1}1)(\forall M\sler T)[\Theta(T, M)\approx_{1} \Phi(T)],
\]
where $x\sler y$ is $\neg (x\rel y)$, i.e.\ $x$ is \emph{nonstandard relative} to $y$.  In other words, in stratified Nonstandard Analysis, the approximation of $\Phi$ and $\varphi$ from $\UWKL$ and $(\exists^{2})$ works \emph{for any object}, not just the standard ones.    
Of course, we have chosen Nelson's framework for this paper, as this approach is more mainstream.   
\end{rem}

\section{The Explicit Mathematics theme for the fan functional}\label{fannypack2}
In this section, we establish the EMT for the \emph{fan functional}, defined as in \eqref{MUC} below, a classically false principle (See Theorem \ref{inco}).
Hence, Corollary \ref{dixie} below implies that the EMT is not limited to statements of classical mathematics.    
For reasons of space, we only establish the EMT for one intuitionistic principle;
In \cite{sambrouwt}, a large number of intuitionistic principles is studied from the point of view of the EMT, including \emph{Brouwer's continuity theorem}.   

\medskip
  
As to its history, the fan functional was introduced by Tait as the first example of a functional which is \emph{non-obtainable}, i.e.\ not computable from lower-type objects (See \cite{noortje}*{p.\ 102}). 
In intuitionistic mathematics, the fan functional emerges as follows:  By \cite{troelstra1}*{2.6.6, p.\ 141}, if a universe of functions $\mathfrak{U}$ satisfies $\bf{EL}+\FAN$, then the class ECF$(\mathfrak{U})$ of \emph{extensional continuous functionals relative to $\mathfrak{U}$}, contains a fan functional.  Here, $\bf{EL}$ is a basic system of intuitionistic mathematics and FAN is the fan theorem, the classical contraposition of WKL.  Similar results are in \cites{troelstra2,troelstra3,gandymahat}.    
\be\label{MUC}\tag{MUC}
(\exists \Omega^{3})(\forall \varphi^{2}) (\forall f^{1}, g^{1}\leq_{1}1 )[\overline{f}(\Omega(\varphi))=_{0}\overline{g}(\Omega(\varphi))\di \varphi(f)=_{0}\varphi(g)].
\ee
Clearly, the existence of the fan functional implies that all type 2-functionals are continuous, which contradicts $(\exists^{2})$ as the latter is equivalent to the 
existence of \emph{discontinuous} functions by \cite{kohlenbach2}*{Prop.\ 3.12}.   

\medskip

Now consider the following principle expressing that all standard type 2 objects are \emph{nonstandard} continuous:  
\be\label{druk}
(\forall^{\st}\varphi^{2})( \forall f^{1},g^{1}\leq_{1}1)\big[ {f}\approx_{1}{g} \di \varphi(f)=_{0}\varphi(g) \big] \tag{$\mathfrak{M}$}.
\ee
We have the following theorem.  
\begin{thm}\label{muckr}
In $\RCAO$, we have $\eqref{MUC}^{\st}\asa \eqref{druk}$.
\end{thm}
\begin{proof}
Let $\{0,1\}^{N}$ be the set of binary sequences of length $N$.  
For $\varphi^{2}$ and $f^{0}\in \{0,1\}^{N}$, we tacitly assume that $\varphi(f)$ stands for $\varphi(f* 00\dots)$.  

\medskip

First of all, assume \eqref{druk} and note that the latter immediately implies
\[
(\forall^{\st}\varphi^{2})(\forall f^{1},g^{1}\leq_{1}1)(\exists^{\st}x^{0})\big[(\overline{f}x=_{0}\overline{g}x) \di \varphi(f)=_{0}\varphi(g) \big].  
\]
Furthermore, we obtain for any fixed $M\in \Omega$, 
\be\label{krudf}
(\forall^{\st}\varphi^{2})(\forall f^{0},g^{0}\in \{0,1\}^{M})(\exists^{\st}x^{0})\big[(\overline{f}x=_{0}\overline{g}x) \di \varphi(f)=_{0}\varphi(g) \big].  
\ee
The formula $\Phi(x,f,g,\varphi)$ in square brackets in \eqref{krudf} is decidable and we define $g(f,g,\varphi,M)$ as the least $x^{0}\leq M$ such that $\Phi(x,f,g,\varphi)$.  As the range of $f,g$ in \eqref{krudf} is discrete, we may compute $\max_{f,g\in \{0,1\}^{M}}  g(f,g,\varphi,M)$.  
However, this \emph{finite} number does not depend on $f $ or $g$ anymore, and we obtain
\be\label{tocheq}
(\forall^{\st}\varphi^{2})(\exists^{\st} y^{0})(\forall f^{0},g^{0}\in \{0,1\}^{M})\big[(\overline{f}y=_{0}\overline{g}y) \di \varphi(f)=_{0}\varphi(g) \big].  
\ee
%
Combining \eqref{druk} and \eqref{tocheq}, we have that
\be\label{cruxsks}\tag{$\mathfrak{P}$}
(\forall^{\st}\varphi^{2})(\exists^{\st}x^{0})(\forall f^{1},g^{1}\leq_{1}1)\big[(\overline{f}x=_{0}\overline{g}x) \di \varphi(f)=_{0}\varphi(g) \big].
\ee
Note that if $y^{0}$ is as in \eqref{tocheq}, then in \eqref{cruxsks} we can take $x=y$.
Now define $\Xi(\varphi^{2},M^{0})$ as the least $y^{0}\leq M^{0}$ as in \eqref{tocheq}, i.e.\ the following `elementary in $\varphi$' functional:
\be\label{dagnoor} 
\Xi(\varphi^{2},M^{0}):=(\mu y\leq M)(\forall f^{0},g^{0}\in \{0,1\}^{M})\big[(\overline{f}y=_{0}\overline{g}y) \di \varphi(f)=_{0}\varphi(g) \big].\tag{$\mathfrak{I}$}
\ee
The functional $\Xi(\varphi^{2},M^{0})$ is clearly $\Omega$-invariant (because we assume \eqref{druk}), i.e.\
\be\label{doytrana}
(\forall^{\st}\varphi^{2})(\forall M^{0},N^{0})[ \Xi(\varphi,M)=_{0}\Xi(\varphi,N)],
\ee
and hence $\eqref{MUC}^{\st}$ follows from \eqref{dagnoor} and \eqref{doytrana} by applying $\Omega$-CA to the latter.   

\medskip

Secondly, assume $\eqref{MUC}^{\st}$ and note that we may assume that $\Omega(\varphi)$ is minimal in that for $m<\Omega(\varphi)$, there are binary sequences $\alpha,\beta$ of length at most $m$ such that $\varphi(\alpha)\ne_{0}\varphi(\beta)$.  Indeed, we need only check a finite number of finite binary sequences to see if $\Omega(\varphi)$ is minimal in this sense.  A simple bounded search can be used to redefine $\Omega(\varphi)$ if necessary.      
Now assume the following formula:
\be\label{maybe}
(\forall^{\st}n^{0},\varphi^{2})\big[ (\forall^{\st}f^{1}\leq_{1} 1)(\varphi(f)=_{0}\varphi(\overline{f}n))\asa n\geq_{0} \Omega(\varphi)  \big].
\ee
As stated in \cite{bennosam}*{\S3.3} and suggested in Remark \ref{tokkiep}, the language $\RCAO$ contains a symbol $\Omega_{0}^{3}$ with defining axiom 
\be\label{totally}
\st(\Omega_{0})\wedge (\forall^{\st}\Xi^{3})\big[M^{\st}(\Xi)\di (\forall^{\st} \varphi^{2})(\Omega_{0}(\varphi)=\Xi(\varphi))\big], 
\ee
where $M(\Omega)$ is the universal formula in \eqref{MUC} with the additional requirement that $\Omega(\varphi)$ is minimal.  
The axiom \eqref{totally} expresses that the fan functional, if it exists, is unique and standard.  
Thus, \eqref{maybe} yields      
\be\label{maybe7}
(\forall^{\st}n^{0},\varphi^{2})\big[ (\forall^{\st}f^{1}\leq_{1} 1)(\varphi(f)=_{0}\varphi(\overline{f}n))\leftarrow n\geq_{0} \Omega_{0}(\varphi)  \big],
\ee
which contains no parameters, i.e.\ \eqref{maybe} qualifies for PF-TP$_{\forall}$ (after bringing the universal quantifier outside the square brackets).  Hence, we obtain:
\be\label{durfje}
(\forall n^{0},\varphi^{2})\big[ (\forall f^{1}\leq_{1} 1)(\varphi(f)=_{0}\varphi(\overline{f}n))\leftarrow n\geq_{0} \Omega_{0}(\varphi)  \big],
\ee
Together with \eqref{maybe}, \eqref{druk} is now immediate.

\medskip

Finally, we prove \eqref{maybe}.  
The reverse direction of the latter is immediate by \eqref{MUC}$^{\st}$;  For the forward direction, assume 
$(\forall^{\st}f^{1}\leq_{1} 1)(\varphi(f)=_{0}\varphi(\overline{f}m_{0}))\wedge  m_{0}<_{0} \Omega(\varphi) $ for some fixed standard $m_{0}^{0}$ and $\varphi^{2}$.  
Fix standard $f^{1},g^{1}\leq_{1}1$ such that $\overline{f}m_{0}=\overline{g}m_{0}$.  We have $\varphi(f)=_{0}\varphi(\overline{f}m_{0})$ and $ \varphi(g)=_{0}\varphi(\overline{g}m_{0})$ by assumption, and
$\varphi(\overline{g}m_{0})=_{0}\varphi(\overline{f}m_{0})$ by extensionality.  However, we now have $\varphi(f)=_{0}\varphi(g)$ for any $f,g$ such that $\overline{f}m_{0}=\overline{g}m_{0}$ while $m_{0}<\Omega(\varphi)$, by assumption.  This contradicts the minimality of $\Omega(\varphi)$, and the forward direction of \eqref{maybe} follows.  
\end{proof}
Similar to \cite{kohlenbach2}*{Prop.\ 3.15}, $\RCAO+\eqref{MUC}$ is a conservative extension of $\WKL_{0}$ by \cite{bennosam}*{Theorem 5}.  
Now consider the formula \eqref{cruxsks} from the proof and consider the following corollary, establishing the EMT for the fan functional.  
\begin{cor}\label{dixie}
In $\RCAO$, we have $\eqref{MUC}\asa \eqref{MUC}^{\st}\asa \eqref{cruxsks}$.
\end{cor}
\begin{proof}
Immediate from the previous proof, in particular \eqref{durfje}, and Remark~\ref{tokkiep}.  
\end{proof}
Unsurprisingly, the fan functional is inconsistent with classical mathematics.
\begin{thm}\label{inco}
The principles $\eqref{MUC}^{\st}$ and $(\exists^{2})^{\st}$ are inconsistent with $\RCAO$.
\end{thm}
\begin{proof}
Let $\varphi_{0}$ be the functional defined by $(\exists^{2})^{\st}$ and let $\Omega$ be as in \eqref{MUC}$^{\st}$.  Now define $f^{1}$ as follows: $f(n)$ is $1$ if $n\geq \Omega(\varphi_{0})$, and zero otherwise,     
and let $\mathbb{1}^{1}$ be the sequence which is $1$ everywhere.  
Then we have $\overline{f}(\Omega(\varphi_{0}))=\overline{\mathbb{1}}(\Omega(\varphi_{0}))$ and hence $\varphi_{0}(f)=\varphi_{0}(\mathbb{1})=1$, by $\eqref{MUC}^{\st}$.  
However, by the definition of $\varphi_{0}$, we have $\varphi_{0}(f)=0$, as clearly $f(\Omega(\varphi_{0})+1)=0$.  
\end{proof}   
Finally, we briefly consider the classically correct \eqref{MUC}$_{0}$ obtainted by limiting \eqref{MUC} to $\varphi^{2}\in \textup{C}(2^{N})$, i.e.\ pointwise continuous\footnote{For completeness, define $\varphi^{2}\in \textup{C}(2^{N})$ as $(\forall f^{1}\leq_{1}1)(\exists N^{0})(\forall g^{1}\leq_{1}1)(\overline{f}N=\overline{g}N\di \varphi(f)=\varphi(g))$.  As usual, denote $\varphi^{2}\in \textup{C}^{\st}(2^{N})$ as the previous formula relative to `st'.} on Cantor space.  
\begin{rem}\rm\label{krem}
As to positive results, \eqref{MUC}$_{0}$ {is} equivalent to the following:
\be\label{krembo2}
(\exists^{\st} \Xi^{3})\big[(\forall^{\st} \varphi^{2}\in \textup{C}(2^{N})) (\forall^{\st} f^{1}, g^{1}\leq_{1} 1 )(\overline{f}\Xi(\varphi)=\overline{g}\Xi(\varphi)\di \varphi(f)=\varphi(g))\big]
\ee
\be\label{krembo}
(\forall^{\st} \varphi^{2}\in \textup{C}(2^{N}))(\exists^{\st}k^{0}) (\forall f^{1}, g^{1}\leq_{1}1 )[\overline{f}k=_{0}\overline{g}k\di \varphi(f)=_{0}\varphi(g)], 
\ee
as `$\varphi^{2}\in \textup{C}(2^{N})$' is internal.  These equivalences are proved as for Theorem \ref{muckr}.  In particular, similar to \eqref{totally}, (the language of) $\RCAO$ contains a symbol $\Xi_{0}^{3}$ and 
\be\label{quark}
\st(\Xi_{0}) \wedge (\forall^{\st}\Gamma^{3})\big[N(\Gamma)\di (\forall^{\st} \varphi^{2}\in C(2^{N}))(\Xi_{0}(\varphi)=\Gamma(\varphi))\big], 
\ee
where $N(\Xi)$ is the square-bracketed formula in \eqref{krembo2} with the additional requirement that $\Xi(\varphi)$ is minimal.   
As to negative results, \eqref{krembo} involves $\textup{C}(2^{N})$ and not $\textup{C}^{\st}(2^{N})$ and it seems impossible to obtain $\eqref{MUC}_{0}\asa \eqref{MUC}_{0}^{\st}$; Indeed, `$\varphi \in \textup{C}^{\st}(2^{N})$' in the latter makes it impossible to apply PF-TP$_{\forall}$.
For the forward implication, we do not have a way of proving that $\varphi^{2}\in \textup{C}(2^{N})$ also yields $\varphi^{2}\in \textup{C}^{\st}(2^{N})$ for standard.  
\end{rem}

\section{The Explicit Mathematics theme around weak K\"onig's lemma}\label{fannypack}
In this section, we establish the EMT for theorems $T$ such that $UT$ is at the level of weak K\"onig's lemma, in line with Conjecture~\ref{braddd}.    
\subsection{The EMT for the fan theorem}\label{EMTFAN}
First of all, we study the \emph{fan theorem}, the classical contraposition of weak K\"onig's lemma, i.e.\ the statement that for all binary trees $T$:
\be\label{genefan}
(\forall  \alpha^{1}\leq_{1} 1)(\exists n^{0})(\overline{\alpha}n\not\in T)\di (\exists k^{0}_{0})(\forall \alpha^{1}\leq_{1} 1)(\exists n^{0}\leq_{0}k_{0})(\overline{\alpha}n\notin T). 
\ee
While weak K\"onig's lemma is universally rejected as `non-constructive' in constructive mathematics, the fan theorem is accepted in intuitionistic mathematics (\cite{brich}).  

\medskip

Denote the principle obtained by the universal closure of \eqref{genefan} by FAN.  For the nonstandard version, let FAN$^{*}$ be FAN$^{\st}$ but with $(\forall \alpha^{1}\leq_{1} 1)$ in the consequent.  
Now, there are at least two possible candidates for the uniform version of the fan theorem, as follows.
\begin{princ}[$\UFAN_{1}$]
There is a functional $\Phi^2$ such that for any binary tree $T$
\[
(\forall \alpha^{1}\leq_{1} 1)(\exists n)(\overline{\alpha}n\not\in T)\di (\forall \alpha^{1}\leq_{1} 1)(\exists n^{0}\leq_{0} \Phi(T))(\overline{\alpha}n\notin T). 
\]
\end{princ}
\begin{princ}[$\UFAN_{2}$]
There is $\Phi^{(1\times 2)\di 0}$ such that for any $T^{1}\leq_{1}1$ and $g^{2}$
\be\label{beestig}
(\forall \alpha^{1}\leq_{1} 1)(\overline{\alpha}g(\alpha)\not\in T)\di (\forall \alpha^{1}\leq_{1} 1)(\exists n^{0}\leq_{0} \Phi(T,g))(\overline{\alpha}n\notin T). 
\ee
\end{princ}
Note that $\UFAN_{2}$ is essentially the BHK-interpretation of intuitionistic logic (See e.g.\ \cite{bridges1}*{p.\ 8}): The functional $g^{2}$ witnesses `how' the tree $T$ has no path, and  the functional $\Phi(T,g)$ has access to this information to determine the finite height of $T$.  For this reason, we refer to $\UFAN_{2}$ as `the' uniform version of FAN.   

\medskip

We have the following preliminary results for the fan theorem.  
Recall that $\RCAO+\eqref{MUC}$ is conservative over $\WKL_{0}$ by \cite{bennosam}*{Theorem 5} and \cite{kohlenbach2}*{Prop.\ 3.15}.  
\begin{thm}\label{shat}
In $\RCAO$, $\FAN^{*}\asa \FAN^{\st}$, $\FAN^{\st}\di \FAN$, $\textup{UFAN}_{1}^{\st}\asa (\exists^{2})^{\st}$, $\eqref{MUC}^{\st}\di \UFAN_{2}^{\st}$. 
\end{thm}
\begin{proof}
By the very structure of FAN$^{\st}$, it is clear that FAN$^{*}$ follows trivially from the latter: If a standard binary tree has finite height, then nonstandard paths are also cut off at this height, as the paths have 
to go through the standard binary sequences of any height.  To prove that $\FAN^{\st}\di \FAN$, assume the former and note that for standard $T^{1}\leq_{1}1$ and standard $g^{2}$, we have
\[
(\forall  \alpha^{1}\leq_{1} 1)(\overline{\alpha}g(\alpha)\not\in T)\di (\exists^{\st} k^{0}_{0})(\forall^{\st} \alpha^{1}\leq_{1} 1)(\exists n^{0}\leq_{0}k_{0})(\overline{\alpha}n\notin T). 
\]
The previous formula trivially implies: 
\[
(\forall  \alpha^{1}\leq_{1} 1)(\overline{\alpha}g(\alpha)\not\in T)\di (\exists^{\st} k^{0}_{0})(\forall \alpha^{1}\leq_{1} 1)(\exists n^{0}\leq_{0}k_{0})(\overline{\alpha}n\notin T), 
\]
and by weakening the consequent we obtain:
\be\label{hanka}
(\forall  \alpha^{1}\leq_{1} 1)(\overline{\alpha}g(\alpha)\not\in T)\di (\exists k^{0}_{0})(\forall \alpha^{1}\leq_{1} 1)(\exists n^{0}\leq_{0}k_{0})(\overline{\alpha}n\notin T), 
\ee
which holds for standard $g^{2}, T^{1}\leq_{1}1$.  However, \eqref{hanka} is an internal formula, say $\varphi(T,g)$, with all parameters shown, and $(\forall^{\st}g^{2}, T^{1}\leq_{1}1)\varphi(T,g)$ implies 
$(\forall g^{2}, T^{1}\leq_{1}1)\varphi(T,g)$ via PT-TP$_{\forall}$.  Using QF-AC$^{1,0}$, FAN is now immediate.    

\medskip

Now assume $(\exists^{2})^{\st}$, fix $M\in \Omega$, and define the functional $\Psi(T,M)$ as $0$ if $(\forall^{\st}n)(\exists \alpha)(|\alpha|=n \wedge \alpha \in T)$, i.e.\ if $T$ is infinite, and as the least 
$k\leq M$ such that $(\forall \alpha^{0}\in \{0,1\}^{*})(|\alpha|= M\di (\exists n^{0}\leq_{0}k)(\overline{\alpha}n\not\in T)$ otherwise.  Clearly, $\Psi(T,M)$ is $\Omega$-invariant (distinguish between finite and infinite trees to see this), and $\textup{UFAN}_{1}^{\st}$ follows.  For the remaining implication, we derive E-$\UWKL^{\st}$ from $\textup{UFAN}_{1}^{\st}$.   
Let $T$ be an infinite standard binary tree and let $\Phi$ be the functional from $\UFAN_{1}^{\st}$.  Recall that $\beta\in T_{\alpha}$ is defined as $\alpha*\beta\in T$.    
Now define $\Psi(T)(1)$ as $0$ if $(\forall \alpha^{0}\in \{0,1\}^{*})(|\alpha|\leq \Phi(T_{1})\di (\exists n^{0}\leq_{0} \Phi(T_{1}))(\overline{\alpha}n\not\in T_{1})$ and $1$ otherwise.
For the general case, define $\Psi(T)(n+1)$ as $0$ if 
\[
(\forall \alpha^{0}\in \{0,1\}^{*})(|\alpha|\leq \Phi(T_{\Psi(T)(n)*1})\di (\exists n^{0}\leq_{0} \Phi(T_{\Psi(T)*1}))(\overline{\alpha}n\not\in T_{\Psi(T)(n)*1}),
\]
and 1 otherwise.  Then $\Psi$ is as required for E-$\UWKL^{\st}$.  

\medskip

For the final implication, define $\Phi(T,g)$ as $\max_{|\sigma|= \Omega(g)\wedge \sigma \leq_{0}1} g(\sigma*00\dots)$.
Alternatively, define $\Psi(T,g,M)$ as follows:
\be\label{funfunctional}
\Psi(T,g,M):=
\begin{cases}
0 & \text{otherwise}  \\
h(T,M) & (\forall \alpha^{0}\in \{0,1\}^{*})(|\alpha|=M \di \overline{\alpha}g(\alpha*00\dots)\not\in T) \\
\end{cases},
\ee
where $h(T,M):=(\mu k\leq M)(\forall \alpha^{0}\in \{0,1\}^{*})(|\alpha|= M\di (\exists n^{0}\leq_{0}k)(\overline{\alpha}n\not\in T)$.
Now use Theorem \ref{muckr}, in particular the nonstandard continuity of $g$, to prove the $\Omega$-invariance of this functional.
\end{proof}


\begin{thm}\label{dfh}
In $\RCAO$, we have $\FAN^{\st}\asa \UFAN_{2}^{\st}$.
\end{thm}
\begin{proof}
The reverse direction is immediate as HAC$_{\INT}$ implies QF-AC$^{1,0}$ relative to `st'.  
To prove that $\FAN^{\st}\di \UFAN_{2}^{\st}$, assume the former and note that for standard $T^{1}\leq_{1}1$ and standard $g^{2}$, we have
\[
(\forall^{\st}  \alpha^{1}\leq_{1} 1)(\overline{\alpha}\tilde{g}(\alpha)\not\in T)\di (\exists^{\st} k^{0}_{0})(\forall^{\st} \beta^{1}\leq_{1} 1)(\exists n^{0}\leq_{0}k_{0})(\overline{\beta}n\notin T). 
\]
The previous formula trivially implies (for any standard $g^{2}, T^{1}\leq_{1}1$) that
\[
(\forall^{\st}  \alpha^{1}\leq_{1} 1)(\overline{\alpha}\tilde{g}(\alpha)\not\in T)\di (\exists^{\st} k^{0}_{0})(\forall \beta^{1}\leq_{1} 1)(\exists n^{0}\leq_{0}k_{0})(\overline{\beta}n\notin T),  
\]
where $\tilde{g}(\alpha)$ is the least $n\leq g(\alpha)$ such that $\overline{\alpha}n\not\in T$.  

We now bring both quantifiers relative to `st' to the front, yielding 
\be\label{ohnepc}
(\exists^{\st}  \alpha^{1}\leq_{1} 1, k^{0}) \big[(\overline{\alpha}\tilde{g}(\alpha)\not\in T)\di (\forall \beta^{1}\leq_{1} 1)(\exists n^{0}\leq_{0}k)(\overline{\beta}n\notin T)\big],
\ee
for any standard $g^{2}, T^{1}\leq_{1}1$.  Note that we could replace the quantifier `$(\exists^{\st}\alpha^{1}\leq_{1}1)$' by a \emph{type 0-quantifier} $(\exists^{\st} \sigma^{0}\leq_{0}1)$.    
In particular, for $\sigma=\overline{\alpha}g(\alpha)$, we have $\sigma=(\sigma*00\dots)\tilde{g}(\sigma*00\dots)\not \in T$.  This will only be relevant for the corollary. 

\medskip
   
Abbreviating the \emph{internal} formula in square brackets in \eqref{ohnepc} by $\psi(\alpha, T,g,k)$, the previous implies  
\be\label{subtle}
(\forall^{\st} g^{2}, T^{1}\leq_{1}1)(\exists^{\st}k^{0}, \alpha^{1}\leq_{1}1)\psi(\alpha, T,g,k),
\ee
and let standard $\Xi^{(1\times 2)\di (0\times 1)^{*}}$ be the functional resulting from applying HAC$_{\INT}$ to \eqref{subtle}.    
Defining $\Phi(T,g):=\max_{i<|\Xi(T,g)(1)|}\Xi(T,g)(1)(i)$, the previous yields
\[
(\forall^{\st} g^{2}, T^{1}\leq_{1}1)(\exists^{\st} \alpha^{1}\leq_{1}1)(\exists k \leq \Phi(T,g) )\psi(\alpha, T,g,k).
\]
Note that we ignored the second component of $\Xi(T,g)$.  Bringing the existential quantifier `$(\exists^{\st} \alpha^{1}\leq_{1}1)$' back inside $\psi$, we obtain for all standard $g^{2}, T^{1}\leq_{1}1$ that
\[
(\forall^{\st}  \alpha^{1}\leq_{1} 1)(\overline{\alpha}g(\alpha)\not\in T)\di (\forall \beta^{1}\leq_{1} 1)(\exists n^{0}\leq_{0}\Phi(T,g))(\overline{\beta}n\notin T),
\]
which yields UFAN$_{2}^{\st}$ and we are done.  
\end{proof}
The following corollary establishes the EMT for $\FAN$ as in the second part of Conjecture \ref{braddd}.   Note that we obtain $UT^{\st}\asa T^{\st}$ without extra assumptions, but require the axiom of choice for the internal version of this equivalence.    
Nonetheless, Hunter notes in \cite{hunterphd}*{\S2.1.2} that any $\QFAC^{\sigma, 0}$ still results in a conservative extension of $\RCA_{0}$.   
\begin{cor}\label{cardargo}
In $\RCAO$, we have $\FAN^{\st}\asa \UFAN_{2}^{\st}\asa \FAN^{*}$.  Adding $\QFAC^{2,0}$, we have $\FAN^{\st}\asa \FAN\asa \UFAN_{2}$.  
\end{cor}
\begin{proof}
We only need to prove the second line in the corollary.  There, the first forward implication follows from the theorem and the final reverse implication is immediate using QF-AC$^{1,0}$.  
For the implication $\FAN\di \UFAN_{2}$, repeat the first part of the proof of the theorem without `st' to obtain \eqref{subtle} without `st'.  
We can make sure the formula $\psi$ is quantifier-free by requiring $|\beta|=k$ in the consequent of \eqref{ohnepc}.  
Furthermore, as noted in the proof of the theorem, the type 1-quantifier in \eqref{ohnepc} can be replaced by a type 0-quantifier.  
Now apply $\QFAC^{2,0}$ to the resulting formula to obtain:
\[
(\exists \Xi^{(2\times 1)\di (0\times 0)})(\forall g^{2}, T^{1}\leq_{1}1)[\psi(\Phi(T,g)(2), T,g,\Xi(T,g)(1))\wedge \Xi(T,g)(2)\leq_{1}1].
\]
and note that $\UFAN_{2}$ follows by ignoring the first component of $\Xi$.  Furthermore, by PF-TP$_{\forall}$, we may assume $\Xi$ is standard;  Hence if for standard $g^{2}, T^{1}\leq_{1}1$ we have $(\forall^{\st}\alpha^{1}\leq_{1}1)(\overline{\alpha}g(\alpha)\not \in T)$, then the tree $T$ is bounded by $\Xi(g,T)(1)$, which is a standard number, i.e.\ UFAN$_{2}^{\st}$ and FAN$^{\st}$ also follow, and we are done.     
\end{proof}
As an exercise, the reader can prove the equivalence between the fan theorem and its \emph{alternative nonstandard version}, defined as:  For all standard $T^{1}$
\be\tag{$\FAN^{**}$}\label{FANSS}
(\forall  \alpha^{1}\leq_{1} 1)(\exists n^{0})(\overline{\alpha}n\not\in T)\di (\exists^{\st} k^{0}_{0})(\forall^{\st} \alpha^{1}\leq_{1} 1)(\exists n^{0}\leq_{0}k_{0})(\overline{\alpha}n\notin T). 
\ee
As a further exercise, the reader can prove the equivalence between the standard part principle \eqref{STP} and \eqref{genefan}$^{\st}$ \emph{for any binary tree}.

\medskip

Similar to the addition of $\QFAC^{2,0}$ in the previous corollary,  certain results in Friedman-Simpson Reverse Mathematics require extra induction (often $I\Sigma_{2}$).  We will often not mention QF-AC$^{2,0}$ in the next section, but leave the associated results implicit.
As shown in \cites{sambrouwt, samzoo}, $\QFAC^{2,0}$ plays a similar important role in the RM of Brouwer's continuity theorem (and related principles) and in the study of uniform versions of principles from the RM zoo.  

\medskip

We finish this section with the following remark.
\begin{rem}\label{stow}\rm
Simpson has the following to say with regard to the mathematical naturalness of logical systems in \cite{simpson2}*{I.12}.
\begin{quote}
From the above it is clear that the five basic systems $\RCA_{0}$, $\WKL_{0}$, $\ACA_{0}$, $\ATR_{0}$, $\FIVE$ arise naturally from investigations of the Main Question. The proof that these systems are mathematically natural is provided by Reverse Mathematics.
\end{quote}
By Corollary \ref{cardargo}, weak K\"onig's lemma is equivalent to the uniform fan theorem UFAN$_{2}$ over a system conservative over $\RCA_{0}$.  
Hence, said uniform principle should also count as mathematically natural.  In the following sections, we shall prove a number of equivalences between weak K\"onig's lemma and uniform principles (involving continuity, Riemann integration, et cetera), 
bestowing mathematical naturalness onto all these higher-order statements.     
\end{rem}
        
\subsection{The EMT for theorems equivalent to weak K\"onig's lemma}
In this section, we establish the EMT for various principles equivalent to weak K\"onig's lemma, including the Heine-Borel lemma, Riemann integration, and the existence of the supremum of continuous functions.  
As noted in Remark \ref{difconc}, these principles can be derived constructively using the fan theorem, in line with Conjecture \ref{braddd}.  
\subsubsection{The Heine-Borel lemma}
We first establish the EMT for the Heine-Borel lemma HB from \cite{simpson2}*{IV.1}.
Careful inspection of the proof in the latter of the equivalence between $\WKL$ and HB, reveals that this proof is uniform.  
Thus, let UHB be the `fully' uniform version of HB, i.e.\ the statement that there is a functional $\Phi^{((0\di 1)\times 2)\di 0}$ such that for all open covers  $I_{n}^{0\di 1}=(c_{n}, d_{n})$ and $g^{2}$, we have:
\[
 (\forall x\in [0,1])(x\in (c_{g(x)}, d_{g(x)}))\di (\forall x\in [0,1])(\exists n \leq \Phi(I_{n}, g ))(x\in (c_{n}, d_{n})).
\]
The functional $g^{2}$ is essential as we otherwise would obtain a version of HB like UFAN$_{1}$, i.e.\ equivalent to $(\exists^{2})$.  
Furthermore, let HB$^{*}$ be HB$^{\st}$, but with the statement that the finite cover covers \emph{all} of $[0,1]$, including the nonstandard reals.  

\medskip

This corollary to Theorem \ref{dfh} establishes the EMT for the Heine-Borel lemma.
\begin{cor}\label{wusup}
In $\RCAO$, we have $\FAN^{\st}\asa \textup{HB}^{\st}\asa \textup{UHB}^{\st}\asa \textup{HB}^{*}$.  Adding $\QFAC^{2,0}$, we have $\FAN^{\st}\asa \textup{HB}\asa \textup{UHB}$.    
\end{cor}
\begin{proof}
For the first two equivalences, the uniformity of the proofs of \cite{simpson2}*{IV.1.1-2} implies UFAN$_{2}\asa \textup{UHB}$, and the equivalence to FAN$^{\st}$ follows from the theorem.  
For the second equivalence, to prove $\HB^{\st}\di \HB^{*}$ is straightforward:  The upper bound $k_{0}$ of the finite cover from {HB}$^{\st}$ also satisfies $(\forall x^{1}\in [0,1])(x\in \cup_{i\leq k_{0}}I_{i})$.  

\medskip

Indeed, suppose $z\in [0,1]$ is such that 
$z\not\in \cup_{i\leq k_{0}}I_{i}$.  Then for all $i\leq k_{0}$, we have either $d_{i}\leq z$ or $z\leq c_{i}$, but we cannot have $z\approx c_{i}$ or $z\approx d_{i}$, as $c_{i}, d_{i}\in [0,1]$ satisfy 
$c_{i}, d_{i}\in \cup_{i\leq k_{0}}I_{i}$, which would also cover $z$.  Thus, fix infinite $M$ and let $i_{0}$ be that $i\leq k_{0}$ such that $[d_{i}](M)<_{0} [z](M)$ and $[d_{i}](M)-[z](M)$ is minimal ($i_{0}$ is the least one if there are several).       
Similarly, let $j_{0}$ be that $j\leq k_{0}$ such that $[c_{j}](M)>_{0} [z](M)$ and $[z](M)-[c_{j}](M)$ is minimal.
By the previous, we have $d_{i_{0}} \not\approx z \not\approx c_{j_{0}}$, implying $d_{i_{0}} +\frac{1}{N} < z < c_{j_{0}}-\frac{1}{N}$ for some finite $N^{0}$.
By the definitions of $d_{i_{0}}$ and $c_{j_{0}}$, there are standard reals in $[d_{i_{0}},c_{j_{0}}]$ which are not covered by $ \cup_{i\leq k_{0}}I_{i}$, a contradiction.  Hence, we must have 
$(\forall x^{1}\in [0,1])(x\in \cup_{i\leq k_{0}}I_{i})$.  
\end{proof}
\subsubsection{Theorems concerning continuity}
In this section, we study establish the EMT for theorems concerning continuity equivalent to weak K\"onig's lemma.

\medskip

The first theorem we consider is the statement `every continuous function on the unit interval is uniformly continuous', which is equivalent to weak K\"onig's lemma by \cite{simpson2}*{IV.2.3}.
As noted in Remark \ref{krem}, the `obvious' approach involving \eqref{MUC}$_{0}$, i.e.\ simply restricting the fan functional to continuous functionals, does not immediately yield an equivalence to the fan theorem.  
Therefore, we will study the following principle, called \textup{MUC}$(\mathfrak{C})$, for various notions of continuity:
\be 
\label{FIN}
(\exists \Theta^{3})(\forall \varphi^{2}\in \mathfrak{C}(2^{N}))(\forall \alpha, \beta \leq_{1}1)(\overline{\alpha}\Theta(\varphi)=\overline{\beta}\Theta(\varphi)\di \varphi(\alpha)=\varphi(\beta)).
\ee
First of all, let $\varphi \in \textup{CC}(2^{N})$ denote that $\varphi\in \textup{C}(2^{N})$ with a modulus of continuity $g_{\varphi}\in \textup{C}(2^{N})$ which in turn has a modulus of continuity $h_{\varphi}$. 
Both moduli are implicitly given together with $\varphi$, and `CC' stands for `constructive continuity'.   
\begin{thm}
In $\RCAo+\QFAC^{2,0}$, we have $\FAN\asa \textup{MUC}(\textup{CC})$.  This equivalence holds relative to `\st' in $\RCAO$.  
\end{thm}
\begin{proof}
For the first reverse implication, let $T$ be a binary tree such that $(\forall \alpha^{1}\leq_{1}1)(\exists n)(\overline{\alpha}n\not\in T)$ and use QF-AC$^{1,0}$ to obtain 
$g^{2}$ such that $(\forall \alpha^{1}\leq_{1}1)(\overline{\alpha}g(\alpha)\not\in T)$.  Define $\tilde{g}(\alpha, T)$ as $(\mu n\leq g(\alpha))(\overline{\alpha}n\not \in T)$ if $\overline{\alpha}g(\alpha)\not \in T$, and zero otherwise.     
By assumption, $\tilde{g}(\cdot, T)$ is continuous as in CC$(2^{N})$;  In particular, this function is its own modulus of continuity.  Applying MUC(CC) yields an uniform upper bound for $\tilde{g}(\cdot, T)$, implying that $T$ is finite, 
and FAN follows.    

\medskip

For the first forward implication, following the proof of \cite{kohlenbach4}*{Prop.\ 4.4}, an associate $\alpha^{1}$ for $\Phi^{2}$ can be defined (uniformly) from $\Phi$ and a continuous modulus of pointwise continuity $g_{\Phi}$.  
By definition, the associate satisfies: 
\be\label{bellow}
(\forall \beta^{1}\leq_{1}1)(\exists k^{0})\alpha(\overline{\beta} k)>0\wedge (\forall \beta^{1}\leq_{1}1, k^{0})(\alpha(\overline{\beta} k)>0 \di \Phi(\beta)+1=\alpha(\overline{\beta} k)).
\ee   
Furthermore, if $g_{\Phi}$ has a modulus of continuity, say $h_{\Phi}$, one easily defines (uniformly in $h_{\Phi}$) a witnessing function $i_{\Phi}$ for $(\forall \beta^{1}\leq_{1}1)(\exists k^{0})\alpha(\overline{\beta} k)>0$, i.e.\ we have 
$(\forall \beta^{1}\leq_{1}1)\alpha(\overline{\beta} i_{\Phi}(\beta))>0$.  Finally, define a tree $T$ by $\sigma\in T\asa \alpha(\sigma)>0$ and apply $\UFAN_{2}$ to obtain the functional from MUC(CC).  
The previous clearly relativizes to the standard world, and HAC$_{\INT}$ implies $\QFAC^{2,0}$ relative to `st'.       
\end{proof}
This result is not satisfying as the CC-notion of continuity is very restrictive.  We therefore study the notion of continuity used in RM in more detail.  
Recall that continuity in the sense of \cite{simpson2}*{II.6.1} amounts to the existence of a modulus of pointwise continuity, i.e.\ the treatment of continuous functions as in RM entails a slight constructive enrichment, which is not problematic for the RM of $\WKL_{0}$ by \cite{kohlenbach4}*{Prop.\ 4.10}.  
We now observe a `nonstandard' enrichment due to the RM-definition of continuity.  This result was first obtained in \cite{samimplicit}.    
\begin{rem}[Continuity]\label{corkom}\rm
In two words, the `nonstandard' enrichment implicit in working with associates is as follows: A standard function which is given by an associate and is continuous relative to standard Cantor space, 
is automatically uniformly continuous \emph{everywhere} there, given weak K\"onig's lemma.  For type~2-functionals, we can only conclude this continuity relative to `st'.  

\medskip

To establish the previous claim, consider a standard function $\alpha^{1}$ such that $(\forall^{\st} \beta^{1}\leq_{1}1)(\exists^{\st} k^{0})\alpha(\overline{\beta} k)>0$, which represents some function $\phi$ on Cantor space.  
In other words, `$\alpha$ is a code for $\phi$' in the sense of \cite{simpson2}*{II.6.1} and one writes symbolically $\phi(\beta)=\alpha(\overline{\beta}(\mu m)(\alpha(\overline{\beta}m)>0))$.  
Now clearly $(\forall \beta^{1}\leq_{1}1)(\exists k^{0}\leq N)\alpha(\overline{\beta} k)>0$ for some standard $N^{0}_{0}$ by $\FAN^{*}$ (See also Corollary \ref{cardargo}) and this implies $(\exists^{\st} N_{0})(\forall \gamma, \beta\leq_{1}1)(\overline{\gamma}N=\overline{\beta}N\di \phi(\gamma)=\phi(\beta))$, i.e.\ $\phi$ is uniformly continuous on \emph{all of} Cantor space.   We also obtain nonstandard continuity as follows:
\be\label{lukas}
(\forall \beta^{1}, \gamma^{1} \leq_{1}1)(\beta\approx_{1}\gamma \di \phi(\beta)=_{0}\phi(\gamma))
\ee
By contrast, repeating the proof of \cite{kohlenbach4}*{Prop.\ 4.10} for standard $\Phi^{2}\in C^{\st}(2^{N})$ relative to `st', we only obtain 
$(\exists^{\st} N_{0})(\forall^{\st} \gamma, \beta\leq_{1}1)(\overline{\gamma}N=\overline{\beta}N\di \Phi(\gamma)=\Phi(\beta))$ 
since we can only obtain the second component of \eqref{bellow} relative to `st'.   
\end{rem}
Hence, we have established that the RM definition of continuity yields a nonstandard enrichment in the form of nonstandard continuity \eqref{lukas}.  
We now study MUC($\mathfrak{C}$) for the RM-definition of continuity, both directly and indirectly.  
\bdefi [Continuity]\label{conki}~  
\begin{enumerate}
\item Let MOD be the statement that for every standard $\Phi^{2}\in \textup{C}(2^{N})$, there is a standard modulus of continuity.  
\item Let ASC be the statement that for every standard $\Phi^{2}\in \textup{C}^{\st}(2^{N})$ and standard $\alpha^{1}$ such that \eqref{bellow}$^{\st}$, we have \eqref{bellow}, i.e.\ a standard associate relative to `st' is also a full associate.     
\item Let $\varphi^{2}\in M(2^{N})$ mean that $\varphi\in \textup{C}(2^{N})$ together with a (continuous) modulus of continuity $g_{\varphi}\in \textup{C}(2^{N})$, given together with $\varphi$.    
\item We write $\phi\in \textup{C}_{\textup{rm}}(2^{N})$ for $\phi$ given by $(\alpha^{1}, g^{2})$ such that $(\forall \beta^{1}\leq_{1}1)\alpha(\overline{\beta}g(\beta))>0$, i.e.\ $\alpha$ is a code for $\phi$ and $g$ is a continuous modulus of continuity of $\phi$.\label{kulllll}
\item Let MUC(C$_{\textup{rm}}$) be \eqref{FIN} modified for $\phi$ coded by $\alpha$ as above.  
\end{enumerate}
\edefi
Note that the modulus in item \eqref{kulllll} does not really constitute an enrichment of the RM-definition of continuity by \cite{kohlenbach4}*{Prop.\ 4.4}.
Furthermore, MOD seems to be a weak principle by \cite{kohlenbach4}*{Prop.\ 4.8}, as the latter shows that the axiom guaranteeing a modulus for every continuous $1\di 1$-functional, is quite weak.   
The study of the nonstandard versions (like \eqref{bang}) in the following theorem is left as an exercise.
\begin{thm}\label{foggart}
In $\RCAO$, we have $\FAN^{\st} \asa  \textup{MUC}(\textup{C}_{\textup{rm}})^{\st}$.  \\
In $\RCAO+\textup{ASC}$, we have $\FAN^{\st} \asa  \textup{MUC}(M)^{\st}\asa \textup{MUC}(M)\asa \textup{MUC}(\textup{C}_{\textup{rm}})$.  \\
In $\RCAO+\textup{ASC}$, we have $[\FAN^{\st} +\textup{MOD}]\asa [ \textup{MUC}(\textup{C})^{\st}+\textup{MOD}]\asa \textup{MUC}(\textup{C})$.  
\end{thm}
\begin{proof}
First of all, we prove the second line in the theorem except for the final forward implication.  The first reverse implication follows as in the proof of the previous theorem.  
The second reverse implication follows from applying PF-TP$_{\forall}$ to MUC($M$) and observing that by ASC, a standard functional $\varphi^{2}\in M^{\st}(2^{N})$ satisfies $\varphi \in M(2^{N})$.  
For the third reverse implication, a continuous modulus uniformly yields an associate by the proof of \cite{kohlenbach4}*{Prop.\ 44}.  

\medskip

For the remaining forward implications, assume FAN$^{\st}$ and note that for standard $\varphi^{2}\in M(2^{N})$, the latter's standard modulus yields a standard associate $\alpha^{1}$ 
as in the proof of \cite{kohlenbach4}*{Prop.~4.4}, i.e.\ we have \eqref{bellow}.
We also have $(\forall^{\st} \beta^{1}\leq_{1}1)(\exists^{\st} k^{0})\alpha(\overline{\beta} k)>0$,    again since $\varphi^{2}$ has a standard modulus of continuity.  Applying FAN$^{*}$ to the latter yields $(\forall \beta^{1}\leq_{1}1)(\exists k^{0}\leq N)\alpha(\overline{\beta} k)>0$ for some standard $N^{0}$.     
Define $\Psi(\varphi, K)$ as 
\[
(\mu k\leq K)(\forall \alpha^{0}, \beta^{0}\leq_{0}1)(|\alpha|=|\beta|=K\wedge \overline{\alpha}k =\overline{\beta}k \di \varphi(\alpha*00\dots)=\varphi(\beta*00\dots)),  
\]
and note that $(\forall^{\st} \varphi^{2}\in M(2^{N}))(\forall L, K\in \Omega)\Psi(\varphi, K)=\Psi(\varphi, L)$.
Since the formula `$\varphi \in M(2^{N})$' is internal, there is (by Corollary \ref{genalli}) a standard $\Theta^{3}$ such that
\be\label{coecke}
(\forall^{\st}\varphi^{2}\in M(2^{N}))(\forall K\in \Omega)\Psi(\varphi, K)=\Theta(\varphi),
\ee
and we have proved \eqref{krembo2} from Remark \ref{krem} for $M$ instead of $\textup{C}$.    
To obtain MUC$(M)$ from this weaker version of \eqref{krembo2}, proceed as in Remark \ref{krem} and the proof of Theorem~\ref{muckr}.  
By ASC, $\varphi^{2}\in M^{\st}(2^{N})$ implies $\varphi \in M(2^{N})$, and MUC$(M)^{\st}$ also follows from the weaker version of \eqref{krembo2}.       

\medskip

Secondly, we prove the first line and the remaining implication in the second line.  For the reverse implication in the first line, define (for a standard binary tree $T$) the function $\alpha^{1}$ as $\alpha(\sigma)=0$ if $\sigma\in T$ and $2$ otherwise.  Applying MUC(C$_{\textup{rm}}$)$^{\st}$ implies that $T$ is bounded if it has no path (all relative to `st').  For the forward implication in the first line, obtain a version of \eqref{krembo2} for 
$\textup{C}_{\textup{rm}}(2^{N})$ instead of $M(2^{N})$ in the same way as the first part of the proof.  
Since FAN$^{\st}$ implies $\alpha^{1}\in \textup{C}^{\st}_{\textup{rm}}(2^{N})\di \alpha^{1}\in \textup{C}_{\textup{rm}}(2^{N})$ for standard $\alpha^{1}$, MUC(C$_{\textup{rm}}$)$^{\st}$ follows from this weak version of \eqref{krembo2}.  
As above, this weak version also implies MUC(C$_{\textup{rm}}$) by PF-TP$_{\forall}$.  

\medskip

Thirdly, we prove the third line.  Assume FAN$^{\st}$ and consider MOD, i.e.\
\be\label{KOD}\textstyle
(\forall^{\st}\Phi^{2}\in \textup{C}(2^{N}))(\exists^{\st}g^{2})(\forall \alpha^{1}, \beta^{1}\leq_{1}1)(\overline{\alpha}g(\alpha)=\overline{\beta}g(\alpha)\di \varphi(\alpha)=\varphi(\beta)).
\ee
As `$\Phi^{2}\in C(2^{N})$' is internal, we may apply HAC$_{\INT}$ to \eqref{KOD}, yielding standard $\Theta^{2\di 2^{*}}$ such that $(\exists g^{2}\in \Theta(\Phi))$ in \eqref{KOD}.  
Now define standard $\Xi^{2\di 2}$ as follows: $\Xi(\Phi)(\alpha^{1}):=\max_{i<|\Theta(\Phi)|}\Theta(\Phi)(i)(\alpha)$.  Clearly $\Xi$ outputs a standard modulus of continuity for standard $\Phi$ as input.    
Now proceed as above to obtain a version of \eqref{coecke} and use ASC to obtain MUC(C)$^{\st}$.  Furthermore, the latter implies FAN$^{\st}$ as in the first part of this proof.  
Next, apply PF-TP$_{\forall}$ to MUC(C) to obtain MOD.  The remaining equivalences follow from the previous parts of the proof.       
%
%
%
\end{proof}
The results in the theorem suggest that we can either directly work with type~1-associates \emph{without additional assumptions}, or work with `representation-free' type~2-functions and adopt additional axioms.  
Since the first route is the one taken in RM, we shall also adopt this approach.  

\medskip

The previous proof reveals a general technique for treating uniform theorems relating to continuity:  
One works with the internal notion of continuity, e.g.\ $\varphi^{2}\in M(2^{N})$ rather than $\varphi \in M^{\st}(2^{N})$, to obtain a version of \eqref{coecke} by Corollary \ref{genalli}.  
Since the standard notion of continuity is included in the internal one (by definition or by ASC), the theorem follows.  
In this light, we shall discuss two more examples of the EMT, namely Riemann integration and supremum for continuous functions.  
\bdefi~
\begin{enumerate}
\item We write 
$y=\sup_{x\in [0,1]}f(x)$ as an abbreviation for:
\be\label{loebak}\textstyle
(\forall x^{1}\in [0,1]) [f(x)\leq y] \wedge  (\forall k^{0})(\exists z^{0}\in [0,1])(y-\frac1k <f(z)).
\ee
\item We write `$\phi\in C_{\textup{rm}}[0,1]$' for $\phi$ given by $(\Phi^{1}, g^{2})$ such that $\Phi$ is a code for $\phi:[0,1]\di \R$ as in \cite{simpson2}*{II.6.1}, and $g$ is a modulus of continuity of $\phi$.
\end{enumerate}
\edefi
Note that the extra modulus in the second part of the definition does not really constitute an enrichment of the RM-definition of continuity by \cite{kohlenbach4}*{Prop.\ 4.4}.
We consider the following principles.
\be\tag{SUP}\label{SUP}  \textstyle
(\forall f\in C_{\textup{rm}}[0,1])(\exists y^{1})(y=\sup_{x\in [0,1]}f(x)).
\ee
\be\tag{USUP}\label{USUP}  \textstyle
(\exists \Phi^{1\di 1})(\forall f\in C_{\textup{rm}}[0,1])(\Phi(f)=\sup_{x\in [0,1]}f(x)).
\ee
\be\tag{SUP$^{*}$}\label{SUP*}  \textstyle
(\forall^{\st} f\in C_{\textup{rm}}[0,1])(\exists^{\st} y^{1})(y=\sup_{x\in [0,1]}f(x)).
\ee

\begin{cor}
In $\RCAO$, $\FAN^{\st}\asa\SUP^{\st}\asa \textup{USUP}^{\st}\asa \SUP^{*}\asa \USUP$.  
\end{cor}
\begin{proof}
The first equivalence follows from \cite{simpson2}*{IV.2.3}.  Now assume FAN$^{\st}$, consider standard $f\in C_{\textup{rm}}[0,1]$ and define $\Psi(f,M)$ as $\max_{i\leq 2^{M}}[f(\frac{i}{2^{M}})](2^{M})$, where $[z](k)=w_{k}$ for a real $z$ represented by the sequence $w_{(\cdot)}^{1}$.  Since by definition also $f\in C^{\st}_{\textup{rm}}[0,1]$, $f$ has a standard supremum $y$ and it is easy to prove that $\Psi(f, M)\approx y \approx \Psi(f,N)$ for $M,N\in \Omega$.  
Applying Corollary \ref{genalli}, there is a standard $\Theta^{1\di 1}$ such that 
\be\label{soepke}
(\forall^{\st}f\in C_{\textup{rm}}[0,1])(\forall N\in \Omega)\big[\Psi(f,N)\approx \Theta(f)   \big],
\ee
and \eqref{soepke} together with the properties of $\Psi(f,M)$ now yields:
\be\label{loeba} \textstyle
(\forall^{\st} f\in C_{\textup{rm}}[0,1])\big[\Theta(f)=\sup_{x\in [0,1]}f(x)\big]^{\st}.
\ee
As in the previous proof, $f\in C^{\st}_{\textup{rm}}[0,1]$ implies $f\in C_{\textup{rm}}[0,1]$, and USUP$^{\st}$ is now immediate from \eqref{loeba}.  
The remaining forward implications are proved as for Theorem \ref{muckr} and Remark \ref{krem}, as the supremum of $f\in C_{\textup{rm}}[0,1]$ is unique.  
In particular, similar to \eqref{totally} and \eqref{quark}, (the language of) $\RCAO$ contains a symbol $\Theta_{0}^{3}$ and 
\[
\st(\Theta_{0}) \wedge (\forall^{\st}\Xi^{1\di 1})\big[O(\Xi)\di (\forall^{\st} f\in C_{\textup{rm}}[0,1])(\Theta_{0}(f)\approx\Xi(f))\big],
\]
where $O(\Theta)$ is \eqref{loeba}.  Now consider $O(\Theta_{0})$ and drop the `st' on the existential quantifier in the second conjunct of \eqref{loebak}$^{\st}$ to obtain a formula of the form $(\forall^{\st} \vx)\varphi(\vx)$ with $\varphi(\vx)$ internal and
without parameters as $\Theta_{0}$ is part of the language of $\RCAO$.  Applying PF-TP$_{\forall}$ now yields USUP and SUP$^{*}$.    
Finally, USUP and SUP$^{*}$ imply SUP$^{\st}$, as can be seen by using the intermediate value theorem.
\end{proof}
With minor adaptation, the proof of the previous corollary also applies to Riemann integration 
Indeed, let INT, UINT, and INT$^{*}$ be {SUP}, {USUP}, and {SUP}$^{*}$, but with \eqref{loebak} replaced by 
`$y=\int_{0}^{1}f(x)\,dx$', which has an obvious definition (\cite{simpson2}*{IV.2.6}).  The following corollary establishes the EMT for Riemann integration.  
\begin{cor}
In $\RCAO$, $\FAN^{\st}\asa \textup{INT}^{\st}\asa \textup{UINT}^{\st}\asa \textup{UINT}\asa \textup{INT}^{*}$.  
\end{cor}
In Remark \ref{corkom}, we showed that the definition of continuity used in RM constitutes a `nonstandard' enrichment in the form of nonstandard continuity \eqref{lukas}.  
Similarly, we now provide an example of a \emph{uniform principle} implicit in \cite{simpson2}*{IV.2.3}, i.e.\ the statement that every continuous function $[0,1]$ is \emph{uniformly} continuous.  
This observation was first made in \cite{samimplicit}.  
\begin{rem}\label{loofer2}\rm
%
%
%
%
First of all, by \cite{kohlenbach4}*{Prop.\ 4.4}, the RM definition of continuity implicitly involves a modulus, and we shall make the latter explicit.  
In other words, we represent a continuous function $\phi$ on Cantor space via a pair of codes $(\alpha^{1}, \beta^{1})$, where $\alpha$ codes $\phi$ and $\beta$ codes its continuous modulus of pointwise continuity $\omega_{\phi}$.      
Thus, $\alpha$ and $\beta$ satisfy $(\forall \gamma^{1}\leq_{1}1)(\exists N^{0})\alpha(\overline{\gamma}N)>0$ and $(\forall \gamma^{1}\leq_{1}1)(\exists N^{0})\beta(\overline{\gamma}N)>0$;  The values of $\omega_{\phi}$ and $\phi$ at $\gamma^{1}\leq_{1}1$, denoted $\omega_{\phi}(\gamma)$ and $\phi(\gamma)$, are $\beta(\overline{\gamma}k)-1$ and $\alpha(\overline{\gamma}k)-1$ for any $k^{0}$ such that the latter are non-negative.  
By the previous:
\be\label{krif2}
(\forall \zeta^{1}, \gamma^{1}\leq_{1}1)(\overline{\zeta}\omega_{\phi}(\zeta)=\overline{\gamma}\omega_{\phi}(\zeta)\di \phi(\zeta)=\phi(\gamma)).
\ee  
Secondly, to represent a \emph{standard} continuous function $\phi$ on Cantor space, we should require that $\phi$ and $\omega_{\phi}$ satisfy the basic axioms $\mathcal{T}_{\st}$ 
(See \cite{bennosam}*{\S2}) of $\RCAO$.  In particular, the numbers $\phi(\gamma)$ and $\omega_{\phi}(\gamma)$ should be standard for standard $\gamma^{1}\leq_{1}1$.  
To this end, we require that $\alpha$ and $\beta$ are standard and that they additionally satisfy:
\begin{align}\label{kruks}
(\forall^{\st} \gamma^{1}\leq_{1}1)(\exists &N^{0})(\exists^{\st}K)[K\geq \alpha(\overline{\gamma}N)>0] \\
&\wedge(\forall^{\st} \gamma^{1}\leq_{1}1)(\exists N^{0})(\exists^{\st}K^{0})[K\geq \beta(\overline{\gamma}N)>0].\notag
\end{align}
Obviously, there are other ways of guaranteeing that $\phi$ and $\omega_{\phi}$ map standard binary sequences to standard numbers, but whichever way we guarantee that $\omega_{\phi}$ and $\phi$ are standard for standard input, \eqref{krif2} yields that 
\be\label{krif3}
(\forall^{\st} \zeta^{1}\leq_{1}1)(\exists^{\st}N^{0})(\forall  \gamma^{1}\leq_{1}1)(\overline{\zeta}N=\overline{\gamma}N\di \phi(\zeta)=\phi(\gamma)),
\ee
since $\omega_{\phi}(\zeta)$ is assumed to be standard for standard binary $\zeta^{1}$.  Combining \eqref{krif3} and \eqref{kruks}, we obtain $(\forall^{\st}\gamma^{1}\leq_{1}1)(\exists^{\st}N)\alpha(\overline{\gamma}N)>0$.    
Applying FAN$^{*}$, which follows from weak K\"onig's lemma by Corollary \ref{cardargo}, we obtain $(\forall \gamma^{1}\leq_{1}1)(\exists N\leq k)\alpha(\overline{\gamma}N)>0$, for some standard $k^{0}$.  
Hence, for every standard and continuous (in the sense of RM) function $\phi$ on Cantor space, we have
\be\label{mattios}
(\exists^{\st}N^{0})(\forall  \zeta^{1}, \gamma^{1}\leq_{1}1)(\overline{\zeta}N=\overline{\gamma}N\di \phi(\zeta)=\phi(\gamma)).  
\ee
given weak K\"onig's lemma (or equivalently \cite{simpson2}*{IV.2.3}) by Corollary \ref{cardargo}.  
In other words, implicit in weak K\"onig's lemma (or again \cite{simpson2}*{IV.2.3}) is the fact that all standard continuous functions are uniformly continuous on \emph{all of Cantor space}.  
The associated statement in the higher type framework is as follows:
\be\label{bang}
(\forall^{\st}\varphi^{2}\in M(2^{N}))(\exists^{\st}N^{0})(\forall  \zeta^{1}, \gamma^{1}\leq_{1}1)(\overline{\zeta}N=\overline{\gamma}N\di \varphi(\zeta)=\varphi(\gamma)).  
\ee
Applying HAC$_{\textup{int}}$ to the previous formula, we obtain $\MUC(M)^{\st}$.  
In conclusion, we have established that the latter \emph{uniform} statement is implicit in the non-uniform statement \cite{simpson2}*{IV.2.3}.  
Similar results hold for other theorems related to continuity, like those concerned with Riemann integration.
%
%
%
\end{rem}

\section{The Explicit Mathematics theme around arithmetical transfinite recursion}\label{EMTATR}
In this section, we establish the EMT for the fourth Big Five system, called $\ATR_{0}$, which formalises \emph{arithmetical transfinite recursion} (\cite{simpson2}*{V}). 
Our theorems and proofs associated with $\ATR_{0}$ show a striking resemblance to those obtained for the fan theorem in Section \ref{EMTFAN}.  
Simpson has previously pointed out a connection between $\WKL_{0}$ and $\ATR_{0}$ in \cite{simpson2}*{I.11.7}, and this connection apparently manifests itself quite strongly at the uniform level.  

\medskip

For reasons of space, we only consider some examples of the EMT around $\ATR_{0}$.  
We will work with the functional version of the latter, which is a mere cosmetic difference.  Indeed, let $\WO(X)$ and $H_{f}(X, Y)$ be the formula $H_{\theta}(X, Y)$ from \cite{simpson2}*{V.1.1 and V.2.2} for $\theta(n^{0}, Y^{1})\equiv (\forall k^{0})[f(k, n, \overline{Y}k)=0]$.   
Then $\ATR_{0}$ in our framework is:
\be\label{ATR}\tag{$\ATR_{\mathbb{o}}$}
(\forall f^{1}, X^{1})[\WO(X)\di (\exists Y^{1})H_{f}(X, Y)].
\ee
Recall that $\WO(X)$ means that the countable linear order $\leq_{X}$ is well-founded.  
Then define the following uniform version of $\ATR_{\mathbb{o}}$ as:
\be\label{UATR}\tag{$\UATR_{\mathbb{o}}$}
(\exists \Phi^{1\di 1})(\forall f^{1}, X^{1})[\WO(X)\di H_{f}(X, \Phi(f, X))],
\ee
and the (non-trivial) nonstandard version of $\ATRO$ as:
\be\label{ATRS}\tag{$\ATR_{\mathbb{o}}^{*}$}
(\forall^{\st} f^{1}, X^{1})[\WO(X)\di (\exists^{\st} Y^{1})H_{f}^{\st}(X, Y)].
\ee
The proof of the following theorem should be compared to that of Theorem \ref{dfh} and Corollary \ref{cardargo}.  
By \cite{yamayamaharehare}*{Theorem 2.2}, the base theory is not stronger than $\ACA_{0}$.   
\begin{thm}\label{scoen}
In $\RCAO+(\exists^{2})$, we have $\ATRO^{\st}\asa \UATRO^{\st}$.  \\
In $\RCAO+\QFAC^{1,1}+(\exists^{2})$, $\ATRO \asa \UATRO\asa \ATRO^{\st}\asa \UATRO^{\st}\asa \ATRO^{*}$.  
\end{thm}
\begin{proof}
The respective uniform principles clearly imply their non-uniform counterparts.  
Furthemore, $\ATRO^{\st}$ implies
\be\label{talfo}
(\forall^{\st} f^{1}, X^{1})(\exists^{\st}Y^{1}, h^{1})[\WO(X, h)^{\st}\di H_{f}(X, Y)^{\st}], 
\ee
where $(\forall h^{1})\WO(X, h)\equiv \WO(X)$, i.e.\ the former is the latter with the only type 1-quantifier brought to the front.  
By \cite{simpson2}*{V.2.3}, the formula in square brackets in \eqref{talfo} is arithmetical (relative to `\st') 
and we may drop all `st' inside the square brackets due to $\paai$, obtained via $(\exists^{2})$. 
Since we now have:
\be\label{crax}
(\forall^{\st} f^{1}, X^{1})(\exists^{\st}Y^{1}, h^{1})[\WO(X, h)\di H_{f}(X, Y)], 
\ee
we apply HAC$_{\textup{int}}$ and obtain a standard $\Psi$ such that 
\[
(\forall^{\st} f^{1}, X^{1})(\exists Y^{1}, h^{1}\in \Psi(f, X))[\WO(X, h)\di H_{f}(X, Y)], 
\]
Since $\Psi$ is standard, we also obtain, ignoring the second component of $\Psi$, that
\[
(\forall^{\st} f^{1}, X^{1})(\exists Y^{1}\in\Psi(f, X)(1) ) (\exists^{\st} h^{1})[\WO(X, h)\di H_{f}(X, Y)].
\]
Since the formula in square brackets is arithmetical, we may again introduce `st' everywhere using $\paai$.  We obtain:
\[
(\forall^{\st} f^{1}, X^{1})(\exists Y^{1}\in\Psi(f, X)(1) ) (\exists^{\st} h^{1})[\WO(X, h)^{\st}\di H_{f}(X, Y)^{\st}], 
\]
which yields by definition that 
\[
(\forall^{\st} f^{1}, X^{1})(\exists Y^{1}\in\Psi(f, X)(1) ) [\WO(X)^{\st}\di H_{f}(X, Y)^{\st}].   
\]
Since $H_{f}(X, Y)^{\st}$ is arithmetical (relative to `\st'), we can use $(\exists^{2})$ to test which entries of $\Psi(f, X)(1)$ satisfy the former.  
Thus, define $\Phi(f, X)$ as $\Psi(f, X)(1)(i_{0})$ where $i_{0}<|\Psi(f, X)(1)|$ is the least number $i^{0}$ such that $\Psi(f, X)(1)(i)$ satisfies $H_{f}(X, \cdot)^{\st}$, if such there is, and the empty set otherwise.  
By definition, we have
\[
(\forall^{\st} f^1, X^{1})[\WO(X)^{\st}\di H_{f}(X, \Phi(f, X))^{\st}].  
\]
Indeed, if $\WO(X)^{\st}$ then by \cite{simpson2}*{Lemma V.2.3} relative to `st', if standard $Z^{1}, W^{1}$ both satisfy $H_{f}(X, \cdot)^{\st}$, then $Z\approx_{1}W$, and $Z=_{1}W$ by $\paai$.  
In other words, there is a \emph{unique standard} $Y^{1}$ satisfying $H_{f}(X, \cdot)^{\st}$ and this $Y^{1}$ is exactly the one computed by $\Phi(f,  X)$ in case $\WO(X)^{\st}$.  

\medskip

Clearly, $\ATRO$ implies \eqref{crax} without `\st'.  In the resulting formula, use $(\exists^{2})$ to make the formula in square brackets quantifier-free, and $\QFAC^{1,1}$ yields:
\be\label{firf}
(\exists \Phi^{1\di 1})(\forall f^{1}, X^{1})[\WO(X, \Phi(f, X)(2))\di H_{f}(X, \Phi(f, X)(1)))].
\ee 
Now $\UATRO$ follows by ignoring the second component of $\Phi$ in \eqref{firf}.  Since the latter does not involve parameters, $\Phi$ is standard by PF-TP$_{\forall}$.   
Thus, \eqref{firf} implies
\be\label{firf2}
(\exists^{\st} \Phi^{1\di 1})(\forall^{\st} f^{1}, X^{1})[(\forall^{\st}h^{1})\WO(X,h)\di H_{f}(X, \Phi(f, X)(1)))].
\ee 
Using $\paai$, $\ATRO^{\st}$ is now immediate.  

\medskip

Next, note that in \eqref{talfo}, we can drop all `st' using $\paai$, except for in $(\forall^{\st}f^{1}, X^{1})$.  
Since the resulting formula has no parameters, we may apply PF-TP$_{\forall}$ to obtain $\ATRO$ (from $\ATRO^{\st}$).  

\medskip

Finally, $\ATRO^{\st}$ clearly implies $\ATRO^{*}$ given $\paai$.  
To obtain $\UATRO$ from $\ATR_{0}^{*}$, drop the `\st' in $(\exists^{\st}Y)$ and apply PF-TP$_{\forall}$.
\end{proof}
\begin{cor}
In $\RCAo+\QFAC^{1,1}$, we have $[\ATRO+(\exists^{2})] \asa \UATRO$.  
\end{cor}
\begin{proof}
Apply $\UATRO$ for the well-order $\{0\}$ to obtain $(\exists^{2})$.  
\end{proof}
The principle CWO from \cite{simpson2}*{V.6.8} has the same syntactical structure as $\ATR_{0}$ by \cite{simpson2}*{V.2.7 and V.2.8}.  Hence, it is straightforward to obtain an equivalence between $\ATRO$ and the (obvious) uniform version of CWO.

\medskip

By Theorems \ref{shat} and \ref{dfh}, it is clear that there is a big difference between the two versions of the uniform fan theorem from Section \ref{EMTFAN}.  
In particular, the inclusion of a realiser for the antecedent of the fan theorem makes a big difference (in logical strength).  We now obtain a similar result
for the statement PST that:
\begin{quote}
\emph{A tree with uncountably many paths has a nonempty perfect subtree}.
\end{quote}
The principle PST is equivalent to $\ATR_{0}$ by \cite{simpson2}*{V.5.5}.  
\begin{princ}[UPST$_{1}$]
There is a functional $\Phi^{1\di 1}$ such that $\Phi(T)$ is a nonempty perfect subtree of any tree $T$ with uncountably many paths.  
\end{princ}\noindent
A tree $T$ is said to have \emph{uncountably many paths} if 
\be\label{kante}
(\forall f_{n}^{0\di 1})(\exists f^{1}\in T)(\forall n^{0})(\exists m^{0})(f(m)\ne f_{n}(m)).  
\ee
\begin{princ}[UPST$_{2}$]
There is a functional $\Phi^{1\di 1}$ such that $\Phi(T, g)$ is a nonempty perfect subtree for any tree $T$ with uncountably many paths, and any $g^{1\di 1}$ witnessing this, i.e.\
$(\forall f_{n}^{0\di 1})(\forall n^{0})(\exists m^{0})\big[g(f_{(\cdot)}(\cdot))(m)\ne f_{n}(m)\wedge g(f_{(\cdot)}(\cdot))\in T\big]$.  
\end{princ}
The following theorem shows that the differences between UFAN$_{1}$ and $\UFAN_{2}$ from Section \ref{EMTFAN} align perfectly with those between $\UPST_{1}$ and $\UPST_{2}$.   
The Suslin functional $(S^{2})$ is the functional version of $\FIVE$, and discussed in Section \ref{zuslin}.  
\begin{thm}\label{mrhin}
In $\RCAo$, we have $\UPST_{1}\asa (S^{2})$.\\ 
In $\RCAo+(\exists^{2})+\QFAC^{1, 1}$, we have $\UPST_{2}\asa \ATRO$.   
\end{thm}
\begin{proof}
The first equivalence follows from \cite{yamayamaharehare}*{Theorem 4.4}.  For the equivalence on the second line, the forward implication is immediate by \cite{simpson2}*{V.5.5}.  
For the remaining implication, assume $\ATRO$ and use \cite{simpson2}*{V.5.5} to obtain PST:
\[
(\forall T^{1})\big[ (\forall f_{n}^{0\di 1})(\exists f^{1}\in T)(\forall n^{0})(\exists m^{0})(f(m)\ne f_{n}(m)) \di (\exists S^{1})P(S, T)\big].
\]
where $ T, S$ are variables ranging over trees, and $P(S,T)$ is the arithmetical formula denoting that $S$ is a non-empty perfect subtree of $T$ (See \cite{simpson2}*{V.4.1}).  
Since $(\exists^{2})$ is available, we may treat arithmetical formulas as quantifier-free.  As is common in RM, we also treat type $0\di 1$-objects as type $1$-objects.  By $\QFAC^{1,1}$, we obtain
\[
(\forall T^{1}, g^{1\di 1})\big[ (\forall f_{n}^{0\di 1})(\forall n^{0})(\exists m^{0})\big[g(f_{(\cdot)}(\cdot))(m)\ne f_{n}(m)\wedge g(f_{(\cdot)}(\cdot))\in T\big] \di (\exists S^{1})P(S, T)\big].
\]  
Bringing the set quantifiers to the front:
\[
(\forall T^{1}, g^{1\di 1})(\exists f_{n}^{0\di 1}, S^{1})\big[(\forall n^{0})(\exists m^{0})\big[g(f_{(\cdot)}(\cdot))(m)\ne f_{n}(m)\wedge g(f_{(\cdot)}(\cdot))\in T\big] \di P(S, T)\big].
\]  
The formula in square brackets is arithmetical, and applying $\QFAC^{1,1}$ yields $\Phi^{1\di 1}$ witnessing the existential quantifiers.
Ignoring the first component of $\Phi$ (involving the witness to $(\exists f_{n}^{0\di 1})$), we obtain for all $T^{1}, g^{1\di 1}$ that
\[
(\forall f_{n}^{0\di 1})(\forall n^{0})(\exists m^{0})\big[g(f_{(\cdot)}(\cdot))(m)\ne f_{n}(m)\wedge g(f_{(\cdot)}(\cdot))\in T\big] \di P(\Phi(T, g), T),
\]  
which is exactly as required, and we are done.  
\end{proof}
In the same way as for Theorem \ref{scoen}, we can establish the following, where PST$_{1}^{*}$ is $\PST^{\st}$ with the `\st' dropped from the antecedent.   
\begin{cor}\label{mrahin}
In $\RCAo$, we have $\UPST_{1}^{\st}\asa \PST_{1}^{*}\asa\UPST_{1}\asa (S^{2})$.\\ 
In $\RCAO+(\exists^{2})+\QFAC^{1, 1}$, $\PST^{\st}\asa \UPST_{2}^{\st}\asa \PST\asa \UPST_{2}\asa \ATRO$.   
\end{cor}
Recall Kohlenbach's heuristic from \cite{kohlenbach2}*{p.\ 293} on the connection between increased logical strength at the uniform level and the essential use of the law of excluded middle in proofs. 
In this light, the behaviour of UPST$_{1}$ is not that surprising as the proof of the perfect set theorem in $\ATR_{0}$ makes use of the law of excluded middle for $\Pi_{1}^{1}$-formulas (See \cite{simpson2}*{p.\ 187}).  

\medskip

On the other hand, the \emph{contraposition} of $\Sigma_{1}^{1}$-separation (\cite{simpson2}*{V.5.1}) has the same syntactical form as PST, and it is possible to obtain an equivalence between $\ATRO$ and the uniform version of this contraposition 
(as in $\UPST_{2}$).  
Further principles with the same syntactical structure as PST are:  The contrapositions of \cite{simpson2}*{V.5.2.2-3} and \cite{simpson2}*{V.6.9.2}, \emph{Ulm's theorem} (\cite{simpson2}*{V.7.3}), \emph{Fra\"iss{\'e}'s conjecture}\footnote{To the best of our knowledge, the exact RM-classification of this theorem is not known.} and certain equivalent principles from \cite{montja}, and the principle TC from \cite{frist}.  
 None of these seem to have nice nonstandard versions as in $T^{*}$ of the EMT.  As to principles with a syntactical structure different from PST, we list \emph{Jullien's theorem} as in \cite{montja} (equivalent to Fra\"iss{\'e}'s conjecture) and the extendibility of $\zeta$, the linear order of the integers, from \cite{denisetal} (equivalent to $\ATR_{0}$).    

\begin{rem}[Mathematical naturalness]\label{stow2}\rm
Recall Remark \ref{stow} concerning the naturalness of logical systems.  Given the above results, $\UATRO$ and related principles also seem to deserve the label `mathematically natural'.  
As to exceptional principles, as well as a potential RM `zoo' (\cite{damirzoo}) between $\ACA_{0}$ and $\ATR_{0}$, one can limit $\ATR_{0}$ to specific well-orders, like the natural numbers.  
\end{rem}
Finally, we suggest further similarities between $\WKL_{0}$ and $\ATR_{0}$.
\begin{rem}[$\WKL_{0}$ versus $\ATR_{0}$]\rm
As noted in Section \ref{STPSEC}, \eqref{STP} is the nonstandard version of weak K\"onig's lemma.  
After Corollary \ref{cardargo}, it is noted that \eqref{STP} is equivalent to $\WKL^{\st}$ generalised to \emph{all} finite trees.  
In light of the similarities between $\WKL_{0}$ and $\ATR_{0}$ pointed out in \cite{simpson2}*{I.11.7}, it is natural question is whether there is a version of the Standard Part principle which corresponds to $\ATR_{0}$.    
Intuitively speaking, $\ATRO^{\st}$ \emph{generalised to any $f$} is equivalent to the statement expressing that one can take the standard part at each step in a (quantifier-free) transfinite recursion, and hand over this standard set to the next step, i.e.\ one can take standard parts along any countable well-order.   This will be explored in future research, as it is beyond the scope of this paper.  
\end{rem}

\section{The Explicit Mathematics theme around $\Pi_{1}^{1}$-comprehension}\label{zuslin}
In this section, we establish the EMT for theorems $T$ such that $UT$ is equivalent to $\Pi_{1}^{1}$-comprehension.   
For reasons of space, we only consider some examples.  Similar to the similarities between the fan theorem and arithmetical transfinite comprehension from the previous section, we establish in Section \ref{sectie} the existence of strong similarities between arithmetical comprehension and $\Pi_{1}^{1}$-comprehension.  

\medskip

We will work with the functional version of $\FIVE$, the so-called Suslin functional (\cite{avi1, yamayamaharehare,kohlenbach2}), defined as follows:
\be\label{suske}
(\exists S^{2})(\forall f^{1})\big[   S(f)=_{0} 0 \asa (\exists g^{1})(\forall x^{0}) (f(\overline{g}x)\ne 0)\big] \tag{$S^{2}$}
\ee
As shown in \cite{bennosam}*{Cor.\ 14}, the Suslin functional $(S^{2})$ is equivalent to:  
\be\label{paaikestoemp2} \tag{$\Paai$}
(\forall^{\st}f^{1})\big[(\forall^{\st} g^{1})(\exists^{\st}x^{0})f(\overline{g}x)=0 \asa (\forall g^{1})(\exists x^{0})f(\overline{g}x)=0\big].
\ee
We now sketch our approach to the EMT around $\Pi_{1}^{1}$-comprehension.  
In particular, we discuss an interesting analogy between $(\exists^{2})$ and $(S^{2})$.
\begin{rem}[Bounded formulas]\label{tola}\rm
As discussed in the proof of Theorem \ref{trikke}, central to the development of the EMT in Section \ref{EMTWKL} 
is that $(\Pi_{1}^{0})^{\st}$-formulas can be replaced by equivalent bounded formulas by simply replacing $(\forall^{\st}k^{0})$ by $(\forall k\leq M)$ for \emph{any} $M\in \Omega$ (assuming of course $T^{*}$).
As will become clear in Section~\ref{sectie}, Nonstandard Analysis also allows us to treat $\Pi_{1}^{1}$-formulas \emph{as bounded formulas}, as will become clear in the following two sections.  
In this way, the EMT for $(S^{2})$ can be established (in Section~\ref{scheidt}) in much the same way as for $(\exists^{2})$,.  
\end{rem}
\subsection{Bounding $\Pi_{1}^{1}$-formulas}\label{sectie}
In this section, we show that $\Pi_{1}^{1}$-formulas are equivalent to natural \emph{bounded} formulas as in \eqref{nonstrat}, assuming $\Paai$.  

\medskip

To this end, consider the following principle.  For finite sequences $\tau^{0},\sigma^{0}$, the notation `$\tau \leq_{0^{*}}\sigma$' is defined as $|\tau|=|\sigma|\wedge (\forall i<|\sigma|)(\tau(i)\leq_{0} \sigma(i))$.    
\begin{princ}[RB] There is a standard functional $\Phi^{1\di 1}$ such that for all standard $f^{1}$ and $M\in \Omega$, we have:
\be\label{nonstrat}
(\forall g^{0}\leq_{0^{*}}\overline{\Phi(f)}M)(\exists x^{0}\leq M)(f(\overline{g}x)=0)\asa (\forall g^{1})(\exists x^{0})(f(\overline{g}x)=0).
\ee
\end{princ}
Intuitively speaking (RB) expresses that it suffices to look for witnesses to $\Sigma_{1}^{1}$-formulas $(\exists g^{1})(\forall x^{0})f(\overline{g}x)\ne0$ below $\Phi(f)$ for standard $f$.  

\medskip

We have the following theorem, where the base theory 
is a conservative extension of $\WKL_{0}$ (See \cite{keisler1, briebenno}).   
\begin{thm}\label{rep}
In $\RCAO+\eqref{STP}$, we have $(S^{2})^{\st}\asa \textup{(RB)}$.
\end{thm}
\begin{proof}
For the reverse implication, it is easy to derive $\paai$ from (RB) using some coding.  Then, the right-hand side of \eqref{nonstrat} may be replaced, using \eqref{STP} and $\paai$, by $(\forall^{\st} g^{1})(\exists^{\st} x^{0})(f(\overline{g}x)=0)$.  
By $\Omega$-CA, there is a standard functional $\Xi(f)$ deciding the truth of the left-hand side of \eqref{nonstrat}, yielding $(S^{2})^{\st}$.     

\medskip

For the forward implication, we recall that $(S^{2})^{\st}\asa \Paai$ by \cite{bennosam}*{Theorem~13}.  
Assuming $\Paai$, the following formula is trivially true:
\be\label{hullb}
(\forall^{\st}f^{1})(\exists^{\st} h^{1})\big[(\forall g^{1}\leq_{1}h)(\exists x)f(\overline{g}x)=0 \di (\forall g^{1})(\exists x)f(\overline{g}x)=0].
\ee
Since the formula in square brackets in \eqref{hullb} is internal, we may apply HAC$_{\textup{int}}$ to obtain standard $\Psi^{1\di 1}$ such that
\[
(\forall^{\st}f^{1})(\exists  h^{1}\in \Psi(f))\big[(\forall g^{1}\leq_{1}h)(\exists x)f(\overline{g}x)=0 \di (\forall g^{1})(\exists x)f(\overline{g}x)=0].
\]
Note that $\Psi(f)$ does not provide a witness to $h$ in \eqref{hullb}, but only a finite sequence of possible witnesses.  
However, we can simply define the standard functional $\Phi^{1\di 1}$ by $\Phi(f)(n):= \max_{i<|\Psi(f)|}\Psi(f)(i)(n)$.  Hence, we obtain   
\[
(\forall^{\st}f^{1})\big[(\forall g^{1}\leq_{1}\Phi(f))(\exists x)f(\overline{g}x)=0 \di (\forall g^{1})(\exists x)f(\overline{g}x)=0], 
\]
and trivially also the reverse implication:
\[
(\forall^{\st}f^{1})\big[(\forall g^{1}\leq_{1}\Phi(f))(\exists x)f(\overline{g}x)=0 \asa (\forall g^{1})(\exists x)f(\overline{g}x)=0].  
\]
Using \eqref{STP} and $\paai$ and the fact that $\Phi$ is standard, we easily obtain:
\be\label{forge}
(\forall g^{0}\leq_{0^{*}}\overline{\Phi(f)}M)(\exists x\leq M)f(\overline{g}x)=0 \asa (\forall g^{1}\leq_{1}\Phi(f))(\exists x)f(\overline{g}x)=0, 
\ee
for any standard $f^{1}$ and $M\in \Omega$.  We now immediately obtain (RB).  
\end{proof}
By the proof of the theorem, the functional $\Phi$ from (RB) is already present in $\RCAO$, but we only obtain (RB) if $\Paai$ is present.  Note that we can repeat the above proof for any special case of $\Paai$.    
\begin{cor}
In $\RCAO+\eqref{STP}+\QFAC^{1,1}$, $(S^{2}) \asa \Paai \asa \textup{(RB)}$.
\end{cor}
\begin{proof}
Immediate from \cite{bennosam}*{Cor.\ 14} and the theorem.
\end{proof}
\begin{rem}[Searching through the reals]\rm
If a $\Sigma_{1}^{0}$-sentence $(\exists n)\varphi(n)$ with $\varphi$ quantifier-free, is known to be true, one need only test $\varphi(0)$, $\varphi(1)$, \dots~to eventually 
find a witness to $(\exists n)\varphi(n)$.  Hence, once can `search through the natural numbers' for a witness to a true $\Sigma_{1}^{0}$-sentence, i.e.\ this infinite search terminates.  
The previous is well-known and it is usually added that `one cannot search through the real numbers (in a similarly basic way)'.  
Nonetheless, (RB) allows us to `search through the reals' for a witness to a $\Sigma_{1}^{1}$-formula as in \eqref{nonstrat} by testing all sequences $\sigma$ such that $|\sigma|=M\wedge (\forall i<M)(\sigma(i)<\Phi(f)(i))$ for $(\forall x\leq M)f(\overline{\sigma}x)\ne0$.  
Now $\Phi$ is already present in $\RCAO$ and if $\Paai$ is given, this search will find a witness.  
\end{rem}
In light of the previous remark, (RB) provides us with a suitable bounding result for $\Pi_{1}^{1}$-formulas, as suggested in Remark \ref{tola}.  We could prove a similar result for $\Delta_{1}^{1}$-comprehension (See \cite{simpson2}*{I.11.8}) 
and the associated Transfer principle, but this is beyond the scope of this paper.    
\subsection{Two examples of the EMT around $\FIVE$}\label{scheidt}
In this section, we provide two examples of the EMT around $\FIVE$.  We make essential use of the fact that $\Pi_{1}^{1}$-formulas can be replaced by bounded ones, as shown in the previous section.

\subsubsection{The \textup{EMT} for $\Sigma_{1}^{0}$-determinacy}
We establish the EMT for the  $\Sigma_{1}^{0}$-determinacy principle, which is equivalent to $\ATR_{0}$ by \cite{simpson2}*{V.8.7}.
We refer to \cite{simpson2}*{V.8} for definitions and notations;  The principle $\Sigma_{1}^{0}$-DET is as follows.
\begin{princ}[$\Sigma_{1}^{0}$-DET] For $\varphi(h^{1},f^{1})\equiv(\exists k^{0})f(\overline{h}k)=0$, we have
\be\label{starbuck8}
(\forall f^{1})\big[(\exists S_{0})(\forall S_{1})\varphi(S_{0}\otimes S_{1},f) \vee (\exists S_{1})(\forall S_{0})\neg\varphi(S_{0}\otimes S_{1},f)  \big].
\ee
\end{princ}
Clearly, if $f$ in $\varphi(S_{0}\otimes S_{1}, f)$ ignores $S_{0}$ in $S_{0}\otimes S_{1}$, then \eqref{starbuck8} just expresses the $\Pi_{1}^{1}$-law of excluded middle.  
The uniform version of $\Sigma_{1}^{0}$-DET is as follows.
\begin{princ}[U$\Sigma_{1}^{0}$-DET] For $\varphi(h^{1},f^{1})\equiv(\exists k^{0})f(\overline{h}k)=0$, we have
\be\label{starbuck9}
(\exists \Phi^{1\di (1\times 1)})(\forall f^{1})\big[(\forall S_{1})\varphi(\Phi(f)(1)\otimes S_{1},f) \vee (\forall S_{0})\neg\varphi(S_{0}\otimes \Phi(f)(2),f)  \big].
\ee
\end{princ}
Finally, let $\Sigma_{1}^{0}$-DET$^{*}$ be $\Sigma_{1}^{0}$-DET$^{\st}$ without `st' in the innermost $\Pi_{1}^{1}$-formulas.  

\begin{thm}\label{doooook8}
In $\RCAO+\eqref{STP}$, we have $\Sigma_{1}^{0}\textup{-DET}^{*}\asa \textup{U}\Sigma_{1}^{0}\textup{-DET}^{\st}\asa (S^{2})^{\st}$.  
\end{thm}
\begin{proof}
We prove the following implications: 
\be\label{gen8}
\Paai \di \Sigma_{1}^{0}\textup{-DET}^{*}\di \textup{U}\Sigma_{1}^{0}\textup{-DET}^{\st} \di (S^{2})^{\st}.
\ee
For the first implication in \eqref{gen8}, $\Paai$ implies $(S^{2})^{\st}$ and the latter implies $\ATRO^{\st}$ and hence $\Sigma_{1}^{0}$-SEP$^{\st}$.  The first implication in \eqref{gen8} is now trivial as \eqref{starbuck8} is a $\Pi_{3}^{1}$-formula.     
The final implication in \eqref{gen8} is also immediate:  For any $f^{1}$, let $\tilde{f}$ be such that in $\varphi(S_{0}\otimes S_{1}, \tilde{f})$, $\tilde{f}$ ignores $S_{0}$ in $S_{0}\otimes S_{1}$, i.e. $\tilde{f}(\overline{S_{0}\otimes S_{1}}k)=f(\overline{S_{1}}\lfloor\frac{k}{2}\rfloor)$.  Then $\Phi(\tilde{f})(2)$ supplies a witness to $(\exists^{\st} g^{1})(\forall^{\st} k^{0})f(\overline{g}k)\ne 0$, if such there is.  Such a functional is known as $(\mu_{1})$ (See \cite{avi2}*{\S8.4.1} and \cite{bennosam}) and implies $(S^{2})$.    

\medskip

For the remaining implication in \eqref{gen8}, we repeat the proof of Theorem \ref{rep} for a particular instance of $\Paai$ provided by $\Sigma_{1}^{0}\textup{-DET}^{*}$.  
The latter can easily be seen to imply $\paai$ and we will treat arithmetical formulas as decidable.  
Furthermore, let $A(f, S_{0}, S_{1})$ be the innermost $\Pi_{1}^{1}$-formula in \eqref{starbuck8}. Then:
\[
(\forall^{\st} f^{1})(\exists^{\st} S_{0}^{1}, S_{1}^{1})\big[A(f,S_{0}, S_{1})\wedge A(f,S_{0}, S_{1})^{\st}\di A(f,S_{0}, S_{1})    \big].
\]
Now let $(\forall S_{2}^{1}, S_{3}^{1})B(S_{2}, S_{3}, f, S_{0}, S_{1})$ be $A(f, S_{0}, S_{1})$, with $B$ arithmetical.  Similar to \eqref{hullb} in the proof of Theorem~\ref{rep}, we obtain 
\[
(\forall^{\st} f^{1})(\exists^{\st} S_{0}^{1}, S_{1}^{1}, h^{1})\big[A(f,S_{0}, S_{1})\wedge (\forall S_{2}^{1}, S_{3}^{1}\leq_{1}h)B(S_{2}, S_{3},f,S_{0}, S_{1})\di A(f,S_{0}, S_{1})    \big].
\]
As in the aforementioned proof, apply HAC$_{\textup{int}}$ to obtain $\Psi$ such that $(\exists S_{0}, S_{1}, h^{1}\in \Psi(f))$ for standard $f^{1}$.  
Define $\Phi^{1\di 1}$ by $\Phi(f)(n):= \max_{1<i<|\Psi(f)|/3}\Psi(f)(3i)(n)$, i.e.\ $\Phi$ ignores $S_{0}, S_{1}$ and computes the maximum of all possible witnesses to $h$ provided by $\Psi$.  
The previous considerations, together with the standardness of $\Phi$, yield that
\[
(\forall^{\st} f^{1})(\exists^{\st} S_{0}^{1}, S_{1}^{1})\big[A(f,S_{0}, S_{1})\wedge (\forall S_{2}^{1}, S_{3}^{1}\leq_{1}\Phi(f))B(S_{2}, S_{3},f,S_{0}, S_{1})\di A(f,S_{0}, S_{1})    \big].
\]
By definition, the inverse implication is again trivial:
\[
(\forall^{\st} f^{1})(\exists^{\st} S_{0}^{1}, S_{1}^{1})\big[A(f,S_{0}, S_{1})\wedge (\forall S_{2}^{1}, S_{3}^{1}\leq_{1}\Phi(f))B(S_{2}, S_{3},f,S_{0}, S_{1})\asa A(f,S_{0}, S_{1})    \big].
\]
Similar to \eqref{forge}, $(\forall S_{2}^{1}, S_{3}^{1}\leq_{1}\Phi(f))B(S_{2}, S_{3},f,S_{0}, S_{1})$ is equivalent to 
\[
(\forall S_{2}^{0}, S_{3}^{0}\leq_{0^{*}}\overline{\Phi(f)}M)\big[(\exists k\leq M)f(\overline{S_{0}\otimes S_{2}}k)=0 \vee  (\forall k\leq M)f(\overline{S_{0}\otimes S_{3}}k)\ne0  \big], 
\]
for standard $f, S_{0}, S_{1}$ and $M\in \Omega$.  Hence, we may treat the innermost $\Pi_{1}^{1}$-formula in $\Sigma_{1}^{0}\textup{-DET}^{*}$ as quantifier-free.  One easily 
verifies that HAC$_{\textup{int}}$ implies $\QFAC^{1,1}$ relative to `st', and applying the latter to $\Sigma_{1}^{0}\textup{-DET}^{*}$ yields $\textup{U}\Sigma_{1}^{0}\textup{-DET}^{\st}$.  
\end{proof}
By \cite{kohlenbach4}*{Cor.\ 4.9} and \cite{yamayamaharehare}*{Theorem 2.2}, $\QFAC^{1,1}$, and hence the base theory of the following theorem, is quite weak.  
\begin{cor}\label{doooook9}
In $\RCAO+\eqref{STP}+\QFAC^{1,1}$, we have 
\be\label{kilopl2}
\Paai\asa \Sigma_{1}^{0}\textup{-DET}^{*}\asa \textup{U}\Sigma_{1}^{0}\textup{-DET}\asa (S^{2}).  
\ee
\end{cor}
\begin{proof}
In \cite{bennosam}*{Cor.\ 15}, the equivalence $\Paai\asa (S^{2})$ is proved.  
The final equivalence in \eqref{kilopl2} follows from the implication $\textup{U}\Sigma_{1}^{1}\textup{-DET}\di (S^{2})$ from the proof of the theorem, as the reverse implication of the equivalence is immediate.    
\end{proof}

\subsubsection{The \textup{EMT} for $\Sigma_{1}^{1}$-separation}
We establish the EMT for $\Sigma_{1}^{1}$-separation, which is equivalent to $\ATR_{0}$ by \cite{simpson2}*{V.5.1}.  The nonstandard version is:  
\begin{princ}[$\Sigma_{1}^{1}$-SEP$^{*}$]
For standard $f_{i}^{1}$ and $\varphi_{i}(n)\equiv (\exists g^{1}_{i})(\forall n_{i})(f_{i}(\overline{g_{i}}n_{i},n)\ne0)$ such that $ (\forall^{\st}n)\neg[ \varphi_{1}^{\st}(n)\wedge \varphi_{2}^{\st}(n)] $, there is standard $Z^{1}$ such that
\be\label{starbuck2}
(\forall n^{0})\big[  \varphi_{1}(n)  \di n\not\in Z \wedge \varphi_{2}(n)\di n\in Z \big].
\ee
\end{princ}
The principle $\Sigma_{1}^{1}$-SEP$^{*}$ states the existence of a separating set for $(\Sigma_{1}^{1})^{\st}$-formulas, but for \emph{all} numbers, not just the standard ones.  
Similarly, $\textup{U}\Sigma_{1}^{1}\textup{-SEP}$ is: 
\begin{princ}[U$\Sigma_{1}^{0}$-SEP] \label{SEP12}
For $\varphi_{i} (n, f)\equiv(\exists g_{i}^{1})(\forall n_{i})(f(\overline{g_{i}}n_{i}, n)\ne0)$, we have
\begin{align}
\big(\exists & F^{(1\times 1\times 0)\di 0}\big)(\forall f^{1},g^{1})\Big[ (\forall n)\neg[ \varphi_{1} (n,f)\wedge \varphi_{2} (n,g)] \di \notag \\ 
& (\forall n^{0})[\varphi_{1} (n,f) \di F(f,g,n)=1  ]\wedge (\forall n) [\varphi_{2} (n,g)\di F(f,g,n)=0 ] \Big].\label{jesestiasep2}
\end{align}
\end{princ}

\begin{thm}\label{doooook2}
In $\RCAO+\eqref{STP}$, we have $\Sigma_{1}^{1}\textup{-SEP}^{*}\asa \textup{E-U}\Sigma_{1}^{1}\textup{-SEP}^{\st}\asa (S^{2})^{\st}$.  
\end{thm}
\begin{proof}
We prove the following implications: 
\be\label{gen2}
\Paai \di \Sigma_{1}^{1}\textup{-SEP}^{*}\di \textup{E-U}\Sigma_{1}^{1}\textup{-SEP}^{\st} \di (S^{2})^{\st}.
\ee
The first implication in \eqref{gen2} is trivial as \eqref{starbuck2} is a $\Pi_{1}^{1}$-formula.  
For the second implication in \eqref{gen2}, clearly $ \Sigma_{1}^{1}\textup{-SEP}^{*}\di  \Sigma_{1}^{0}\textup{-SEP}^{*}$, i.e.\ we may use $\paai$ by Theorem \ref{doooook}.  
Next, note that $ \Sigma_{1}^{1}\textup{-SEP}^{*}$ implies for all standard $f_{1}^{1}, f_{2}^{1}$ that
\be\label{test2}
(\forall^{\st}n)[ \neg\varphi_{1}^{\st}(n, f_{1})\vee \neg\varphi_{2}^{\st}(n, f_{2})]\di  (\forall n)[\neg \varphi_{1}^{}(n, f_{1})\vee \neg\varphi_{2}^{}(n, f_{2})].
\ee
Similar to the proof of Theorem \ref{doooook} (involving the functions $f_{3},f_{4}$), this yields:  
\be\label{test4}
(\forall^{\st}f_{1}^{1},f_{2}^{1}, n^{0} )\Big[ [ \neg\varphi_{1}^{\st}(n, f_{1})\vee \neg\varphi_{2}^{\st}(n, f_{2})]\di [\neg \varphi_{1}^{}(n, f_{1})\vee \neg\varphi_{2}^{}(n, f_{2})]\Big],
\ee
\be\label{test3}
(\forall^{\st}f_{1}^{1},f_{2}^{1}, n^{0} )\Big[ [ \neg\varphi_{1}^{\st}(n, f_{1})\wedge \neg\varphi_{2}^{\st}(n, f_{2})]\di  [\neg \varphi_{1}^{}(n, f_{1})\wedge \neg\varphi_{2}^{}(n, f_{2})]\Big].
\ee
Now, the consequent (resp.\ antecendent) of both \eqref{test4} and \eqref{test3} is a $\Pi_{1}^{1}$-formula (resp.\ relative to `st').  
In other words, the previous centered formulas are instances of $\Paai$ and we can repeat the proof of Theorem \ref{rep} for the conjunction of \eqref{test4} and \eqref{test3}.  
This yields the existence of a standard functional $\Phi$ such that for all standard $f_{1}, f_{2}, n$ and $M\in \Omega$, we have
\begin{align}
[\neg\varphi_{1}^{M,\Phi}(n, f_{1})\vee \neg\varphi_{2}^{M,\Phi}(n, f_{2})]&\asa [\neg \varphi_{1}^{}(n, f_{1})\vee \neg\varphi_{2}^{}(n, f_{2})] \notag\\
&\wedge\label{evidenze} \\
[\neg\varphi_{1}^{M,\Phi}(n, f_{1})\wedge \neg\varphi_{2}^{M,\Phi}(n, f_{2})]&\asa [\neg \varphi_{1}^{}(n, f_{1})\wedge \neg\varphi_{2}^{}(n, f_{2})] \notag
\end{align}
where we abbreviate $\neg\varphi_{i}^{M,\Phi}(f,n)\equiv [(\forall g_{i}^{0}\leq_{0*}\Phi(f))(\exists x_{i}\leq M)f(\overline{g_{i}}x, n)=0]$.  
Now define the functional $\Psi$ as follows:
\be\label{full}
\Psi(f_{1},f_{2},M)(n):=
\begin{cases}
~0 &~~\varphi_{1}^{M,\Phi}(n,f_{1}) \wedge \neg\varphi_{2}^{M,\Phi}(n, f_{2})\\
~1 & \neg\varphi_{1}^{M,\Phi}(n, f_{1}) \wedge~ ~\varphi_{2}^{M,\Phi}(n, f_{2})\\
~2 &  \neg\varphi_{1}^{M,\Phi}(n, f_{1}) \wedge \neg\varphi_{2}^{M,\Phi}(n, f_{2})\\
~3 & ~~ \varphi_{1}^{M,\Phi}(n, f_{1}) \wedge ~~\varphi_{2}^{M,\Phi}(n, f_{2})\\
\end{cases}.
\ee
%
%
Using \eqref{evidenze}, it is now straightforward (by following the proof of Theorem \ref{doooook}) that $\Psi$ is both as required for $\textup{U}\Sigma_{1}^{1}\textup{-SEP}^{\st}$ and $\Omega$-invariant, in case the standard $f_{i}$ satsify $(\forall^{\st}n)[ \neg\varphi_{1}^{\st}(n, f_{1})\vee \neg\varphi_{2}^{\st}(n, f_{2})]$.  We also show this explicitly now.  

\medskip

Indeed, for standard $n$, if $\varphi_{1}^{\st}(n, f_{1})$ then $\varphi_{1}(n,f_{1})$ by $\paai$;  By assumption and \eqref{test2}, we have $\neg\varphi_{2}(n, f_{2})$, and the first case in \eqref{full} holds by \eqref{evidenze} (for any infinite $M$). 
If $\varphi_{2}^{\st}(n, f_{2})$ holds for standard $n$, then similarly $\neg\varphi_{1}(n, f_{1})$ by the previous, and the second case in \eqref{full} holds (for any infinite $M$).  
Since $\varphi_{2}^{\st}(n, f_{2})\wedge \varphi_{1}^{\st}(n, f_{1})$ is impossible by assumption, the final case in \eqref{full} does not occur.  If for some standard $n_{0}$ the third case holds, we have it for all $M\in \Omega$ by the second 
conjunct of \eqref{evidenze}. 
Hence, if the third case in \eqref{full} occurs, it does so for all $M\in \Omega$.  

\medskip  

As $\Omega$-CA requires quantification over \emph{all} standard sequences $f_{i}^{1}$ as in \eqref{tomega23}, we need to specify the behaviour when the separation assumption $(\forall^{\st}n)[ \neg\varphi_{1}^{\st}(n, f_{1})\vee \neg\varphi_{2}^{\st}(n, f_{2})]$ is not met.      
Thus, let $\Xi(f_{1},f_{2},n, M)$ be the $\Omega$-invariant characteristic function of the left-hand side of the first conjunct in \eqref{evidenze}, and let $\Lambda(f_{1}, f_{2}, n)$ be its standard part obtained via $\Omega$-CA.  
Now define $\Theta(f, g, M)$ as $\Psi(f,g,M)$ if $(\forall n\leq M)\Lambda(f, g,n)=1$, and $0$ otherwise.    
Using $\paai$ and $ \Sigma_{1}^{0}\textup{-SEP}^{*}$ as in the previous paragraph, it is clear that $\Theta(T,M)$ is $\Omega$-invariant, i.e.\ we have
\be\label{tomega23}
(\forall^{\st}n , f^{1},g^{1})(\forall N,M\in \Omega)\big[ \Theta(f,g,M)\approx_{1}\Theta(f,g,N) \big]. 
\ee  
The axiom $\Omega$-CA provides a standard functional $\Phi(\cdot)\approx_{1}\Theta(\cdot, M)$ which satisfies $\textup{U}\Sigma_{1}^{1}\textup{-SEP}^{\st}$.  
As to the standard extensionality of $\Phi$, note that if $\varphi_{i}^{\st}(n, f_{i})$, i.e.\ in one the first two cases of \eqref{full}, this extensionality property is immediate due to \eqref{test}.  By the latter, the third case of \eqref{dolpi} also does not occur for standard $h_{1}, h_{2}$ such that $h_{i}\approx_{1}f_{i}$.  For the final case in \eqref{dolpi}, a similar argument involving the functions $f_{3}, f_{4}$ from Theorem \ref{doooook} guarantees standard extensionality.    

\medskip

For reasons of space, the proof of U$\Sigma_{1}^{0}$-SEP $\di (S^{2})$ is left to the reader.   
\end{proof}
By \cite{kohlenbach4}*{Cor.\ 4.9} and \cite{yamayamaharehare}*{Theorem 2.2}, $\QFAC^{1,1}$, and hence the base theory of the following theorem, is quite weak.  
\begin{cor}\label{doooook3}
In $\RCAO+\eqref{STP}+\QFAC^{1,1}$, we have 
\be\label{kilopl}
\Paai\asa \Sigma_{1}^{1}\textup{-SEP}^{*}\asa \textup{U}\Sigma_{1}^{1}\textup{-SEP}\asa (S^{2}).  
\ee
\end{cor}
\begin{proof}
In \cite{bennosam}*{Cor.\ 15}, the equivalence $\Paai\asa (S^{2})$ is proved.  
The final equivalence in \eqref{kilopl} follows from the implication $\textup{U}\Sigma_{1}^{1}\textup{-SEP}\di (S^{2})$ from the proof of the theorem, as the reverse implication of the latter is immediate.    
\end{proof}
In light of the uniformity of the proof in \cite{denisetal}*{\S4}, the uniform version of the extendibility of $\zeta$, the linear order of the integers, seems equivalent to uniform $\Sigma_{1}^{1}$-separation, and the same for the nonstandard versions.   

\medskip

Finally, it should be possible to formulate a version of Conjecture \ref{braddd} for $\ATR_{0}$ and $\FIVE$ after studying more examples of the EMT around $\FIVE$.    

\begin{ack}\rm
This research was supported by the following funding bodies: FWO Flanders, the John Templeton Foundation, the Alexander von Humboldt Foundation, and the Japan Society for the Promotion of Science.  
The author expresses his gratitude towards these institutions. 
The author would like to thank Ulrich Kohlenbach, Karel Hrbacek, Benno van den Berg, Steffen Lempp, Paul Shafer, Mariya Soskova, and Denis Hirschfeldt for their valuable advice.  
\end{ack}

\bye